\documentclass[a4paper, 10pt]{amsart}

\usepackage{amsmath}
\usepackage{amssymb}
\usepackage{amsthm}
\usepackage{mathtools, verbatim}
\usepackage[utf8]{inputenc}
\usepackage{indentfirst}
\usepackage{enumerate}
\usepackage[colorlinks=true]{hyperref}
\hypersetup{urlcolor=blue, citecolor=red}
     
\usepackage{graphicx}

\usepackage[usenames,dvipsnames]{pstricks}
\usepackage{epsfig}
\usepackage{pst-grad} 
\usepackage{pst-plot} 

\usepackage{ifthen}
\usepackage{color, xcolor}

\def\sideremark#1{\ifvmode\leavevmode\fi\vadjust{\vbox to0pt{\vss
 \hbox to 0pt{\hskip\hsize\hskip1em
 \vbox{\hsize2.1cm\tiny\raggedright\pretolerance10000
  \noindent #1\hfill}\hss}\vbox to15pt{\vfil}\vss}}}%

\allowdisplaybreaks

\usepackage{fullpage}
\numberwithin{equation}{section}

\def\polhk#1{\setbox0=\hbox{#1}{\ooalign{\hidewidth\lower1.5ex\hbox{`}\hidewidth\crcr\unhbox0}}}

\def\Xint#1{\mathchoice
{\XXint\displaystyle\textstyle{#1}}%
{\XXint\textstyle\scriptstyle{#1}}%
{\XXint\scriptstyle\scriptscriptstyle{#1}}%
{\XXint\scriptscriptstyle\scriptscriptstyle{#1}}%
\!\int}
\def\XXint#1#2#3{{\setbox0=\hbox{$#1{#2#3}{\int}$ }
\vcenter{\hbox{$#2#3$ }}\kern-.6\wd0}}

\def\dashint{\Xint-}

\newcommand{\dist}{\operatorname{dist}}

\newcommand{\R}{\mathbb{R}}
\newcommand{\eps}{\varepsilon}

\theoremstyle{plain}
\newtheorem{Theorem}{Theorem}[section]

\newtheorem{Lemma}[Theorem]{Lemma}

\newtheorem{Proposition}[Theorem]{Proposition}
\newtheorem{Remark}[Theorem]{Remark}

\begin{document}


\title{ Spectral partition problems with volume and inclusion constraints}

\author{P\^edra D. S. Andrade, Ederson Moreira dos Santos, Makson S. Santos and Hugo Tavares}

\maketitle
	
\begin{abstract}

 In this paper we discuss a class of spectral partition problems with a measure constraint, for partitions of a given bounded connected open set. We establish the existence of an optimal open partition, showing that the corresponding eigenfunctions are locally Lipschitz continuous, and obtain some qualitative properties for the partition. The proof uses an equivalent weak formulation that involves a minimization problem of a penalized functional where the variables are functions rather than domains, suitable deformations, blowup techniques and a monotonicity formula.

\end{abstract}
\medskip


\medskip 


\section{Introduction}
In this article we study an optimal partition problem with volume and inclusion constraints, for a cost functional depending on the first Dirichlet eigenvalue of the Laplacian. Let $\Omega $ be a bounded connected open set of $\mathbb{R}^N$,  for $N\geq2$. For an integer $k\geq 2$ and $0<a<|\Omega|$, we consider the multiphase shape optimization problem
\begin{equation}\label{eigenvalue_problem}
\inf \left\{\sum_{i=1}^k \lambda_1(\omega_i)\;\Big|\;
\begin{array}{c}
\omega_i \subset \Omega \mbox{ are nonempty open sets for all } i=1,\ldots, k, \vspace{0.05cm}\\
 \omega_i \cap \omega_j = \emptyset  \: \text{for all}\: i \not=j \mbox{ and } \sum_{i=1}^{k}|\omega_i| =  a\\
\end{array}
\right\},
\end{equation}
where $\lambda_1(\cdot)$ denotes the first  Dirichlet eigenvalue and $|\cdot|$ stands for the Lebesgue measure.  The main goal of this paper is to  prove the existence of an optimal open partition to \eqref{eigenvalue_problem}, showing also that the corresponding eigenfunctions are locally Lipschitz continuous (see Theorem \ref{thm:main} below). The proof uses a weak formulation that involves a minimization problem of a penalized functional where the variables are functions rather than domains.
\begin{figure}[!h]
\begin{center}
 \includegraphics[width= 5cm]{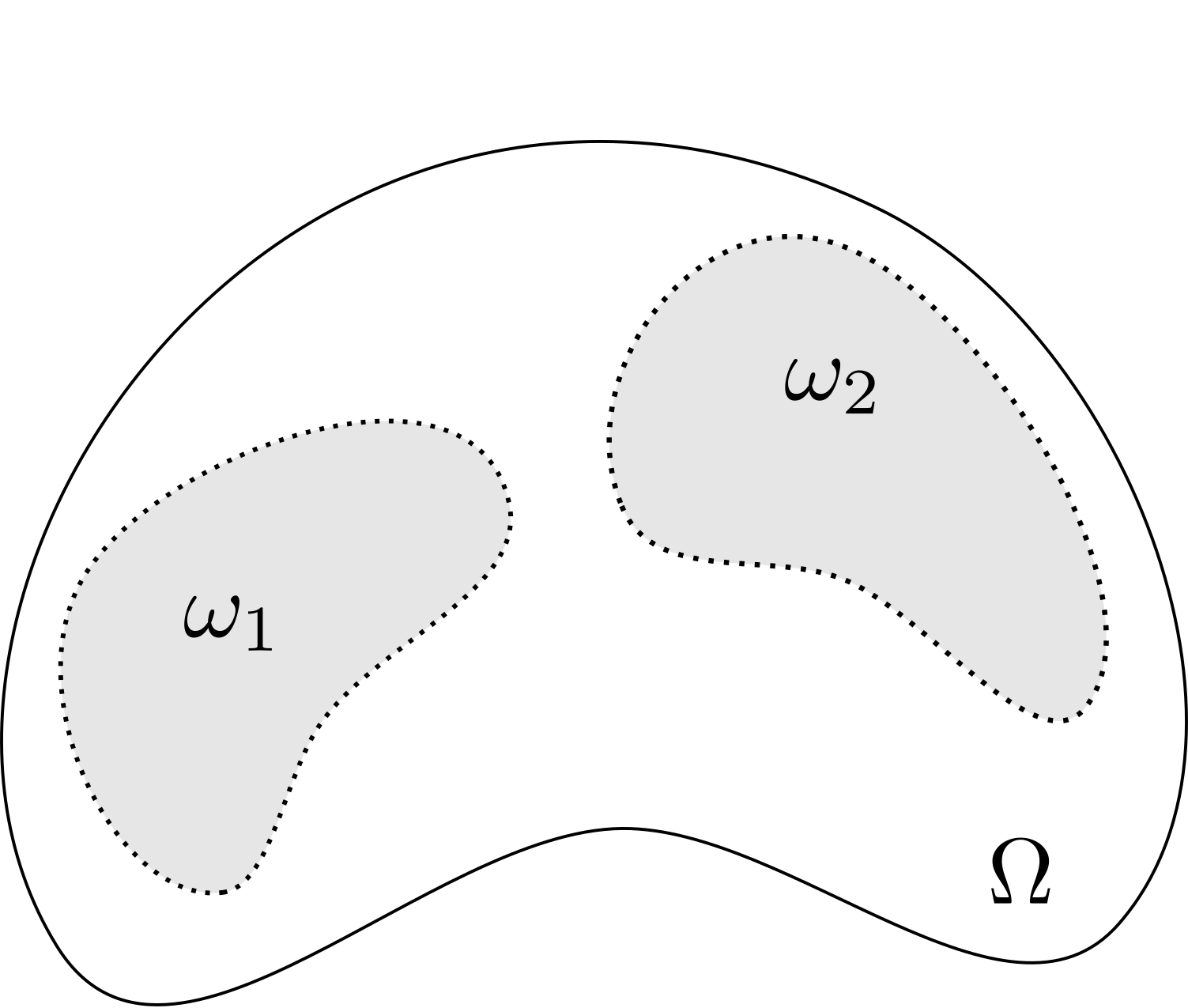}   
\end{center}
\caption{An admissible partition $(\omega_1,\omega_2)$ in $\mathcal{P}_a(\Omega)$ for a certain $0<a < |\Omega|$ and $k=2$. }
\end{figure} 

Minimizing a functional with measure constraints appears in electromagnetic casting processes \cite{Crouzeix1991,Descloux1994}. The study of these functionals is motivated by their applications to industry and has attracted many mathematicians, physicists, and engineers. In particular, minimization problems with volume constraint involving the Dirichlet eigenvalues for the Laplace operator have been extensively studied by many authors; we refer to \cite{BucurButtazzoHenrot, BMPV2015, BZ1995, ButtazzoDalMaso, BD1993} and the references therein, as well as the books \cite{BB2005, Henrot2006,  Henrot2017, HenrotPierre2005, Velichkov2015}. 

\smallbreak

Problem \eqref{eigenvalue_problem} with one phase corresponds to 
 \begin{equation}\label{first_problem}
  \inf \{ \lambda_1(\omega), \: \omega \subset \Omega,\: \omega \: \text{open}, \: |\omega|= a\}.
 \end{equation}
Let us split the discussion of such a problem between the case when $\Omega$ is a bounded domain and when $\Omega=\R^N$.

For problem \eqref{first_problem} with $\Omega$ a bounded domain, due to the lack of a suitable topology, the classical variational techniques are not appropriate to prove the existence and regularity of Dirichlet eigenvalues problems with measure constraints. One important notion that aids in dealing with this issue is the $\gamma$-convergence, introduced by G. Dal Maso and U. Mosco in \cite{DalMasoMosco1986, DalMasoMosco1987}. This type of convergence allowed G. Buttazzo and G. Dal Maso in \cite[Example 2.6]{BD1993} to produce the first and classical existence result (actually, the authors prove it in a very general situation which includes the minimization of $\lambda_k$). The authors prove that there exists a minimizer to \eqref{first_problem} in the class of quasi-open sets. Such a class of sets is, in fact, the largest family for which the Dirichlet eigenvalues of the Laplacian problem is still well-posed and inherits a strong maximum principle. For more details on quasi-open sets, see \cite{BB2005}. 

To obtain \emph{open} optimal sets is more challenging than having quasi-open ones (actually, there are even situations where an open solution does not exist; see, for instance, \cite[Theorem 3.11]{Hayouni1999}).  Then, the fundamental question is to understand whether and when a solution has additional regularity properties. In \cite{BHP2005}, T. Brian\c con, M. Hayouni and M. Pierre  prove the existence of an open solution $\omega$ to \eqref{first_problem} when $\Omega$ is bounded, establishing at the same time locally Lipschitz regularity for the corresponding eigenfunctions. An important part of the strategy in \cite{BHP2005} is to take a solution $\bar u$ to the problem
\[
\inf \left\{\int_\Omega |\nabla u|^2:\ u\in H^1_0(\Omega),\ \int_\Omega u^2=1,\ |\{u\neq 0\}|\leq a\right\}
\]
and show that it is also a minimizer of the penalized functional
\begin{equation}\label{eq:penalizedfunctional}
J(u)=\int_\Omega |\nabla u|^2+\lambda_a \left(1-\int_\Omega u^2\right)^+ + m (|\{u\neq 0\}|-a)^+,\quad \text{ where } \lambda_a:=\int_\Omega |\nabla \bar u|^2,
\end{equation}
for $m$ sufficiently large. In the end, the authors show the equivalence between these two problems, and
the equivalence between \eqref{first_problem} and  
\begin{equation}\label{first_problem2}
  \min \{ \lambda_1(\omega)+m(|\omega|-a)^+ ,\ \omega\subset \Omega \text{ open}  \}
 \end{equation}
for large $m>0$. Then, T. Brian\c con and J. Lamboley, in \cite{BrianconLamboley2009}, prove that any open solution $\omega$ has a locally finite perimeter and that, up to a negligible set, $\partial \omega\cap \Omega$ is analytic. Finer results about the singular set are shown in \cite{MazzoleniTerraciniVelichkov2017} (which deals with a more general vectorial case).

\begin{Remark}
We observe that \cite{BHP2005} also deals with Dirichlet energies related to problems of type $-\Delta u=f$, $u\in H^1_0(\omega)$. The regularity of the free boundary for such optimization problems is addressed in \cite{Briancon2004} by T. Brian\c con. When the operator is of divergence type, M. Hayouni (see \cite{Hayouni1999}) shows the existence and Lipschitz continuity under the assumption that the state function is positive by increasing the admissible set and regularizing the volume constraint, respectively. We observe that, for a problem like \eqref{first_problem} with eigenvalues of a divergence-type operator, E. Teixeira and S. Snelson obtain, in \cite{Teixeira-Snelson2022}, H\" older regularity of the eigenfunctions; they prove this when the diffusion coefficient is close, in a suitable sense, to the identity. The authors in \cite{Hayouni1999, Teixeira-Snelson2022} use different approaches. Finally, E. Russ, B. Trey and B. Velichkov in \cite{RussTreyVelichkov2019} perform a complete study of problem \eqref{first_problem} for eigenvalues of elliptic operators with drift.
\end{Remark}

The study of the one phase problem \eqref{first_problem} with $\Omega=\R^N$ corresponds to the problem appearing in the 19th century in the monograph \cite{Rayleigh1945}. By the Faber-Krahn inequality \cite{Faber1923, Krahn1925}, it is classical to show that the ball of volume $a$ is the minimizer of \eqref{first_problem}.  For the minimization problem of the second eigenvalue, it is known that the solution is a union of two disjoint balls with equal measure, by the Hong-Krahn-Szeg\"o inequality \cite{Hong,Krahn1926}. Using different strategies, D. Bucur in \cite{Bucur2012} and, more or less simultaneously, D. Mazzoleni and A. Pratelli in \cite{Mazzoleni-Pratelli2013} obtain a general existence result in the whole space $\mathbb{R}^N$ for the minimization of the $k$-th eigenvalue with a prescribed measure in the class of quasi-open sets.  For more details on the spectral problem see \cite[Chapters 2 \& 3]{Henrot2017}.

\smallbreak

Concerning now the multiphase case \eqref{eigenvalue_problem}, our work is, up to our knowledge, the first to treat the case when $\Omega$ is a bounded set. There are, however, related problems, which we now describe. When $\Omega=\R^N$ it is easy to check, again by Faber-Krahn inequality, that the solution is a union of $k$ disjoint balls (see for instance the proof of Theorem \ref{th-bolas} below). If, instead, one is minimizing $L_\ell(\omega_1,\ldots,\omega_k)=\sum_{i=1}^k \lambda_\ell(\omega_i)$, for $\ell \geq 3$ then, up to our knowledge, nothing is known (for $\ell=2$, the solution is a union of $2k$ disjoint balls).

Observe that, by a scaling argument (reasoning, for instance, as in \cite[Proposition 6.3]{Velichkov2015}), such problems are equivalent to
\begin{equation}\label{BV_problem2}
  \min \left\{ \sum_{i= 1}^{k} \lambda_{\ell}(\omega_i) + m |\omega_i|\: : \: \omega_i \subset \R^N, \: \omega_i  \: \text{open},\: \omega_i\cap \omega_j = \emptyset \text{ for $i\neq j$} \right\}, 
\end{equation}
for some $m>0$.  A similar problem but with partitions contained in a bounded domain $\Omega$ is studied in
 \cite{BucurVelichkov2014} by D. Bucur and B. Velichkov. More precisely, the authors treat the following minimization problem
\begin{equation}\label{BV_problem}
  \min \left\{ \sum_{i= 1}^{k} \lambda_{\ell}(\omega_i) + m |\omega_i|\: : \: \omega_i \subset \Omega, \: \omega_i  \: \text{quasi-open},\: \omega_i\cap \omega_j = \emptyset \text{ for $i\neq j$} \right\}, 
\end{equation}
where $m>0$, $\Omega$ is a given bounded open set. Notice that, since $m >0$, the solution will be a  partition of $\Omega$ with an empty region because, in general, the sets $ \omega_i$ will not cover $\Omega$.  Hence, the geometry of the partitions in \cite{BucurVelichkov2014} is similar to the geometry of the partitions in $\mathcal{P}_a(\Omega)$. The authors prove qualitative properties for an optimal partition, such as inner density estimates, finite perimeter and absence of triple points,among others. For $\ell=1,2$, they also show the existence of open minimizers. A complete study of the regularity of the free boundary of optimal sets is obtained for $\ell=1$ in \cite[Corollary 1.3]{DePhilippisSpolaorVelichkov2021}. For results in the special case $N=2$ and numerical simulations, see \cite{BogoselVelichkov2016}.

We emphasize that, when $\Omega$ is bounded, problems \eqref{eigenvalue_problem} and \eqref{BV_problem} for $\ell=1$ are not equivalent (scaling arguments no longer work), and it seems that some deformation arguments used in \eqref{BV_problem} do not provide directly useful information for
 \eqref{eigenvalue_problem}, due to the positive part and $a>0$ in our measure constraint. 
  \smallbreak

Observe that, when $a=|\Omega|$, problem \eqref{first_problem} becomes a spectral partition problem without volume constraint, namely 
\begin{equation}\label{OPPclassical}
 \inf \left\{\sum_{i=1}^k \lambda_1 (\omega_i)\;\Big|\;
\begin{array}{c}
\omega_i \subset \Omega \mbox{ are nonempty open sets for all } i, \vspace{0.05cm}\\
 \omega_i \cap \omega_j = \emptyset  \:  \text{for all}\: i \not=j \: \mbox{ and } \: \overline\Omega = \cup_{i= 1}^k \overline{\omega_i}\\
\end{array}
\right\}.
\end{equation}
Combining the results from \cite{CaffarelliLin2007,ContiTerraciniVerziniOPP} (see also \cite[Section 8]{TavaresTerracini1}), it is known that optimal partitions exists, and the free boundary $\cup_i \partial \omega_i$ is, up to a singular set of lower dimension, regular. Finer results for the singular set are proved in the recent paper \cite{Alper}, namely that the $(N-2)$--Hausdorff dimension of the singular set is finite, together with a stratification result. The case of higher eigenvalues has been addressed in \cite{RamosTavaresTerracini} (see also references therein).

\smallbreak

To conclude this literature review, we refer the papers  \cite{STTZ2018, STZ2021}, where the authors consider the same cost functional as \eqref{eigenvalue_problem}, with a distance constraint between elements of each partition (namely $\dist(\omega_i, \omega_j) \geq  r$ for every $i\neq j$) instead of a measure constraint.

\subsection{ Statement of the main results and structure of the paper}

As already mentioned, the main goal of this paper is to prove the existence of (open) minimizers to \eqref{eigenvalue_problem}, together with the local Lipschitz continuity of the corresponding eigenfunctions. For that, it is convenient to relax the measure constraint, dealing with 
\begin{equation}\label{eigenvalue_problem2}
c_a =\inf_{(\omega_1, \ldots, \omega_k) \in {\mathcal P}_a(\Omega)} \sum_{i=1}^k \lambda_1(\omega_i),
\end{equation}
where \begin{equation*}
{\mathcal P}_a(\Omega):= \left\{(\omega_1, \ldots, \omega_k)\;\Big|\;
\begin{array}{c}
\omega_i \subset \Omega \mbox{ are nonempty open sets for all } i, \vspace{0.05cm}\\
 \omega_i \cap \omega_j = \emptyset  \: \text{for all}\: i \not=j \mbox{ and } \sum_{i=1}^{k}|\omega_i| \leq  a\\
\end{array}
\right\}.
\end{equation*}
We do not pass through the notion of quasi-open sets; instead,  one important feature to study the existence and regularity of the solutions in this scenario is the equivalence between optimal partition problems and minimization problems involving a state functional.  We introduce a weak formulation that involves a cost functional, where the variables are functions rather than domains, namely 
\begin{equation}\label{main_func}
\tilde c_a = \inf_{(u_1, \ldots, u_k) \in H_a} J(u_1, \ldots, u_k),
\end{equation}
where
\begin{equation*}
J(u_1, \ldots, u_k) :=\sum_{i=1}^k \int_{\Omega} |\nabla u_i|^2
\end{equation*}
and
\begin{equation*}
H_a : = \Big\{(u_1, \ldots,  u_k)  \: \Big|\:  
  u_i \in H^1_0(\Omega) \mbox{ and }\int_{\Omega} u_i^2 = 1 \:\: \mbox{for every $i$}, \\
  \int_{\Omega} u_i^2 u_j^2 =0\  \mbox{ for $i\neq j$},\ \  \sum_{i =1}^k|\Omega_{u_i}|\leq  a 
\Big \},
\end{equation*}
with $\Omega_{u_i} := \{x \in \Omega \: |\: u_i(x)\neq 0\}$ for all $i \in \{  1, \ldots, k\}$.

Our main result is the following.
\begin{Theorem}\label{thm:main} 
The problem \eqref{eigenvalue_problem2} admits a solution. Moreover:
\begin{enumerate}[i)]
\item Given any optimal partition $(\omega_1,\ldots, \omega_k)\in \mathcal{P}_a(\Omega)$, then each $\omega_i$ is connected and $\sum_{i=1}^k |\omega_i|=a$. Therefore, the problems \eqref{eigenvalue_problem} and \eqref{eigenvalue_problem2} have the same solutions.  In addition, if $u_i$ is a first eigenfunction associated with $\omega_i$, then $u_i$ is locally Lipschitz continuous in $\Omega$.
\item Problems \eqref{eigenvalue_problem2} and \eqref{main_func} are equivalent in the following sense:\begin{enumerate}[a)]
\item  $c_a=\tilde c_a$; 
\item if $(u_1,\ldots, u_k)\in H_a$ is an optimal solution of \eqref{main_func} and $\Omega_{u_i}:=\{u_i\neq 0\}$, then  $(\Omega_{u_1},\ldots, \Omega_{u_k})\in {\mathcal P}_a(\Omega)$ solves \eqref{eigenvalue_problem2}; 
\item  if $(\omega_1,\ldots, \omega_k)\in \mathcal{P}_a(\Omega)$ is an optimal partition for \eqref{eigenvalue_problem2} and $u_i$ is a first eigenfunction associated to the set $\omega_i$, then $(u_1,\ldots, u_k)\in H_a$ is a minimizer for \eqref{main_func}.
\end{enumerate}
\end{enumerate}
\end{Theorem}

We prove the existence of an optimal partition by exploiting the equivalence between the problems \eqref{eigenvalue_problem2} and \eqref{main_func}, which plays a crucial role in overcoming technical difficulties to treat \eqref{eigenvalue_problem2} directly. For instance, by using the direct method of calculus of variations to \eqref{main_func}, we can easily prove the existence of minimizers. Since we are not working with the concept of quasi-open sets, we notice that the continuity of minimizers is a fundamental hypothesis in proving such equivalence. This is the content of Proposition \ref{prop:equivalence_levels} and Proposition \ref{thm:condition_C1}. 

In order to prove the continuity of minimizers, we adapt the techniques presented in \cite{BHP2005}.  However, several difficulties appear due to the fact we are dealing with partitions instead of only one set. Firstly, the generalization of \eqref{eq:penalizedfunctional} that works in our scenario is:

\begin{equation*}
J_{\mu} (u_1,\ldots, u_k):= \sum_{i=1}^k \frac{\displaystyle \int_{\Omega}|\nabla u_i|^2}{\displaystyle \int_{\Omega} u_i^2}  + \mu \left[ \sum_{i=1}^k |{\Omega}_{u_i}| - a\right ]^{+}, \qquad \text{for} \ \ (u_1\ldots, u_k) \in \overline{H},
\end{equation*}
where
\begin{equation*}
\overline{H} : = \Big\{(u_1,\ldots,  u_k) \in H^1_0(\Omega;\R^k)\:  \Big|\: \quad  u_i \neq 0\ \forall i,\   u_i \cdot u_j \equiv 0\ \ \forall i\neq j \Big \}.
\end{equation*}

To extract information from this new penalized energy, we rely on some deformation arguments from \cite{ContiTerraciniVerziniAsymptotic, ContiTerraciniVerziniOPP, ContiTerraciniVerziniVariation} (see Appendix \ref{appendix}). These were introduced for the study of the spectral partition \emph{without} volume constraints, like \eqref{OPPclassical} above. In the context of problem \eqref{OPPclassical}, these deformations provide that any solution should satisfy a set of inequalities (namely, they should belong to the class $\mathcal{S}_{\lambda_1, \ldots, \lambda_k}$, see \eqref{S-class} below). Due to the presence of an empty region (related to the fact that $a<|\Omega|$), in our context, we obtain more complex inequalities, see Proposition \ref{prop_03} below, which is the key result in our paper.

Throughout Section \ref{secLip}, where we prove Lipschitz continuity, we use the continuity of the minimizers proved in Section \ref{sec:regularity}, namely in the proof of Proposition \ref{prop:m2}. 
For the proof of Lipschitz continuity, we were not able to apply directly the ideas in \cite{BHP2005} to our framework. In our case, we proceed as in \cite{ContiTerraciniVerziniAsymptotic, ContiTerraciniVerziniOPP, ContiTerraciniVerziniVariation} by using powerful tools such as blow-up methods, the Caffarelli-Jerison-Kenig monotonicity formula, suitable inequalities obtained via deformations (Proposition \ref{prop_03}), and some properties of the class $\mathcal{S}_{\lambda_1, \ldots, \lambda_k}$ mentioned before.

\smallbreak

In the next result, we characterize the minimizers of \eqref{eigenvalue_problem}, in the case we have enough space inside $\Omega$.

\begin{Theorem}\label{th-bolas}
There exists $\bar a=\bar a(\Omega,N,k)$ such that, for $a<\bar a$, then any solution of \eqref{eigenvalue_problem} is a partition made of $k$ disjoint open balls, all with the same radius.\end{Theorem}

Finally, for the case with $k=2$, we prove the existence of an optimal partition that inherits some symmetry from the box $\Omega$.

\begin{Theorem}\label{th-simetria}
Consider \eqref{eigenvalue_problem} with $k=2$ and suppose $\Omega \ni 0$ is axially with respect to a unit vector $e$. Then there exists an optimal partition $(\omega_1, \omega_2)$ for \eqref{eigenvalue_problem}, with corresponding nonnegative eigenfunctions $u_1, u_2$, such that:
\begin{enumerate}[a)]
\item  $\omega_1$ and $\omega_2$ are axially symmetric with respect to $e$;
\item $u_1$ and $u_2$ are foliated Schwarz symmetric with respect to $e$ and $-e$, respectively.
\end{enumerate}
\end{Theorem}

\smallbreak

This paper is structured as follows: Section \ref{sec_not} is devoted to the equivalence between the minimization problems  \eqref{eigenvalue_problem2} and \eqref{main_func}. Also, we prove the existence of the minimizers to \eqref{eigenvalue_problem2} and \eqref{main_func} and define a penalized functional. In Section \ref{sec:regularity}, we introduce some properties for the  class  $\mathcal{S}_{\lambda_1, \ldots, \lambda_k}$ and show that the minimizers of \eqref{main_func} are bounded and continuous functions. Section \ref{secLip} is dedicated to establishing the Lipschitz continuity for the corresponding eigenfunctions. In Section \ref{Sec_proof}, we introduce the proofs of the main results. In Section \ref{appendix}, we present some properties for some classes of deformations and gather some auxiliary results that are used in the manuscript.   

To conclude, we point out that our strategy based on variations is flexible and can be applied in other contexts; a work regarding the study of \eqref{eigenvalue_problem} with eigenvalues associated to divergence type operators is currently in preparation.

\section{Preliminaries and existence of minimizers for the weak formulation}\label{sec_not}

In this part, we show the existence of minimizer to \eqref{main_func} and introduce a key step for the proof of Theorem \ref{thm:main}, namely the equivalence of problem \eqref{main_func} with a penalized version $J_{\mu}$ defined below. We start by showing the equivalence (assuming that the minimizers are continuous functions) between problems \eqref{eigenvalue_problem2} and \eqref{main_func}. It is worth highlighting that throughout this paper, for an open set $\omega \subset \Omega$ and a function $u:\omega \to \mathbb{R}$, we also denote by $u$ its extension to $\Omega$ as being zero outside $\omega$.

\begin{Proposition}\label{prop:equivalence_levels}
It holds that $\tilde c_a\leq c_a$. Moreover, if there exists $(\bar{u}_1, \ldots, \bar{u}_k) \in {H}_a$ a minimizer to \eqref{main_func} and $\bar{u}_1, \ldots, \bar{u}_k$ are continuous in $\Omega$, then $(\Omega_{{\bar u}_1}, \ldots, \Omega_{{\bar u}_k}) \in \mathcal{P}_a(\Omega)$ is a minimizer of the functional in \eqref{eigenvalue_problem2}, each $\bar{u}_i$ is a first Dirichlet eigenfunction in $\Omega_{\bar{u}_i}$ and $\tilde c_a=c_a$.
\end{Proposition}
\begin{proof}
We start by choosing a partition $(\omega_1,\ldots, \omega_k) \in   {\mathcal P}_a(\Omega)$. For each $i \in \{ 1, \ldots, k\}$, consider $u_i$ the first (positive) Dirichlet eigenfunction corresponding to $\omega_i \subseteq \Omega$, normalized in $L^2(\omega_i)$. Then $(u_1,\ldots, u_k)\in H_a$ and
\begin{align}
\tilde c_a
&\leq  \displaystyle \sum_{i=1}^k\int_{\Omega} |\nabla u_i|^2
 =  \displaystyle \sum_{i=1}^k\int_{\omega_{i}} |\nabla u_i|^2
= \sum_{i=1}^k\lambda_1(\omega_{i}). \label{eq:equivalence}  
\end{align}
Applying the infimum in \eqref{eq:equivalence} over the set ${\mathcal P}_a(\Omega)$, we get $\tilde c_a\leq  c_a$.

Now, assume there exists $(\bar{u}_1, \ldots, \bar{u}_k) \in H_a$ which solves the minimization problem \eqref{main_func}, and that $\bar{u}_1, \ldots, \bar{u}_k$ are continuous in $\Omega$. Observe that, for each $i$, $\Omega_{\bar u_i}=\{\bar u_i\neq 0\}$ is an open set and that $\int_{\Omega} |\nabla \bar u_i|^2=\int_{\Omega_{\bar u_i}} |\nabla \bar u_i|^2\geq \lambda_1(\Omega_{\bar u_i})$ and $(\Omega_{\bar u_1},\ldots, \Omega_{\bar u_k})\in  {\mathcal P}_a(\Omega)$. Then
\begin{align*}
\tilde c_a &=J(\bar u_1,\ldots, \bar u_k) = \sum_{i=1}^k\int_\Omega |\nabla \bar{u}_i|^2\geq \sum_{i=1}^k\lambda_1(\Omega_{\bar u_i})\geq c_a.
\end{align*}
Then $\tilde c_a=c_a$ and $ c_a$ is achieved.
\end{proof}

Therefore, in order to prove the main result of this paper, namely Theorem \ref{thm:main}, it is sufficient to prove:
\begin{equation*}\tag{C1}\label{eq:condition1}
\text{There exists a minimizer $(\bar{u}_1, \ldots, \bar{u}_k) \in H_a$ to \eqref{main_func}.}
\end{equation*}
\begin{equation*}\tag{C2}\label{eq:condition2}
\text{If $(\bar{u}_1, \ldots, \bar{u}_k) \in H_a$ is a minimizer to \eqref{main_func}, then each $u_i$ is locally Lipschitz continuous in $\Omega$.}
\end{equation*}

The first condition is proved next in Proposition  \ref{thm:condition_C1}, while the second is shown in Sections \ref{sec:regularity} and \ref{secLip}.

\begin{Proposition}\label{thm:condition_C1}
The infimum $\tilde c_a$ is achieved, that is, there exists $(\bar{u}_1,\ldots, \bar{u}_k) \in H_a$ such that
\[
\tilde c_a = J(\bar{u}_1,\ldots, \bar{u}_k) \leq  J(u_1,\ldots, u_k)\; \; \text{ for all}\;\; (u_1,\ldots,u_k) \in H_a.
\]
\end{Proposition}
\begin{proof} 
 It is clear that  $0\leq \tilde c_{{a}} <\infty$.
Let $(u_{1,n},\ldots, u_{k,n})_{n\in \mathbb{N}} \subset H_a$ be a minimizing sequence for $J$ and we may suppose that
\begin{equation}\label{inq_existence}
\tilde c_a+ 1 \geq J( u_{1,n},\ldots, u_{k,n})=\sum_{i=1}^k \int_{\Omega}|\nabla u_{i,n} |^2\quad \text{for all } \, n.
\end{equation}
Therefore, for all $i=1,\ldots, k$, the sequences $(u_{i,n})_{n \in \mathbb{N}}$ are bounded  in $H^1_0(\Omega)$ and there exist $\bar u_i\in H^1_0(\Omega)$ such that, up to subsequences, as $n\to \infty$,
\[
u_{i,n} \to \bar u_i\qquad \text{weakly in $H^1_0(\Omega)$,\ strongly in $L^2(\Omega)$, \ a.e. in $\Omega$.}
\]
Then
\[
\int_{\Omega} \bar{u}_i^2=  \lim_{n \to \infty} \int_{\Omega} u_{i,n}^2 = 1 \text{ for every $i$},\quad |\bar{u}_i(x)|^2 |\bar{u_j}(x)|^2=  \lim_{n\to \infty} |u_{i,n}(x)|^2 |u_{j,n}(x)|^2=0 \:\text{ a.e. $x$ in } \Omega,\ i\neq j
\]
and, by applying Fatou's Lemma, we obtain
\[
|\Omega_{\bar{u}_i}| = \int_{\Omega} \chi_{\Omega_{\bar{u}_i}}(x)
\leq \int_{\Omega} \liminf_{n \rightarrow \infty} \chi_{\Omega_{u_{i,n}}}(x)\leq \liminf_{n \rightarrow \infty}\int_{\Omega} \chi_{\Omega_{u_{i,n}}}(x) = \liminf_{n \rightarrow \infty} |\Omega_{u_{i,n}}|.
\]
Thus,
\[
\sum_{i=1}^k |\Omega_{\bar{u}_i}|  \leq \liminf_{n \rightarrow \infty}  \sum_{i=1}^k |\Omega_{u_{i,n}}| 
\leq a.
\]
Therefore $(\bar{u}_1,\ldots, \bar{u}_k)$ belongs to $H_a$ and so
\[
\tilde c_a \leq J(\bar{u}_1,\ldots,\bar{u}_k)\leq \liminf_{n\to \infty} J(u_{1,n},\ldots, u_{k,n})=\tilde c_a,
\]
which finishes the proof.
\end{proof}

Given $\mu>0$, define the penalized functional $J_{\mu}: \overline{H} \to \mathbb{R}$ by
\begin{equation*}
J_{\mu} (u_1,\ldots, u_k):= \sum_{i=1}^k \frac{\displaystyle \int_{\Omega}|\nabla u_i|^2}{\displaystyle \int_{\Omega} u_i^2}  + \mu \left[ \sum_{i=1}^k |{\Omega}_{u_i}| - a\right ]^{+}, \qquad \text{for} \ \ (u_1\ldots, u_k) \in \overline{H},
\end{equation*}
where
\begin{equation*}
\overline{H} : = \Big\{(u_1,\ldots,  u_k) \in H^1_0(\Omega;\R^k)\:  \Big|\: \quad  u_i \neq 0\ \forall i,\   u_i \cdot u_j \equiv 0\ \ \forall i\neq j \Big \}.
\end{equation*}

\begin{Proposition}\label{prop_pen}
Let $(\bar{u}_1,\ldots, \bar{u}_k)\in H_a$ be a minimizer of \eqref{main_func} and take $\mu> \left(\dfrac{2^{(k-1)/2}\, \tilde{c}_a}{N | B_1|^{1/N}}a^{\frac{2-N}{2N}} \right)^2$. Then
\begin{equation}\label{eq_pen}
\sum_{i=1}^k\int_{\Omega}|\nabla \bar{u}_i|^2  \leq J_\mu(u_1,\ldots, u_k)
\end{equation}
for all $(u_1,\ldots, u_k) \in \overline{H}$. In particular, $J$ and $J_{\mu}$ have a common minimizer.
\end{Proposition}

\begin{proof} 
Let $u_1,\ldots, u_k \in H^1_0(\Omega)\setminus \{0\}$ such that $u_i \neq 0$ for all $i$,  $ u_i\cdot u_j =0$ for all $i\neq j$ and $ \sum_{i=1}^k |{\Omega}_{u_i}| \leq a$. Since $\left(\frac{u_1}{\|u_1\|_2}, \ldots, \frac{u_k}{\|u_k\|_2}\right)\in H_a$ and $(\bar u_1,\ldots, \bar u_k)$ is a minimizer of \eqref{main_func},
\begin{equation}\label{eq_01}
\sum_{i=1}^k\int_{\Omega}|\nabla \bar{u}_i|^2  \leq \sum_{i=1}^k\frac{\int_{\Omega}|\nabla u_i|^2}{\int_{\Omega} u_i^2}.
\end{equation}

By the compactness embedding of $H^1_0(\Omega)$ into $L^2(\Omega)$ and reasoning exactly as in the proof of Proposition \ref{thm:condition_C1}, we can find a minimizer $(u_{1,\mu},\ldots, u_{k, \mu})\in \overline{H}$ of $J_{\mu}$. With no loss of
 generality, we may assume $u_{i,\mu}\geq0$ and $\|u_{i,\mu}\|_2 = 1$ for all $i$. 
 
Let $\mu> \left(\dfrac{2^{(k-1)/2}\, \tilde{c}_a}{N | B_1|^{1/N}}a^{\frac{2-N}{2N}} \right)^2$. Suppose, by contradiction, that $ \sum_{i=1}^k |{\Omega}_{u_{i,\mu}}| > a$ and consider the auxiliary functions $u_i^t:= (u_{i,\mu} - t)^{+}$, for $t>0$. Observe that, since $(u_{1,\mu},\ldots, u_{k,\mu}) \in \overline{H}$, we also have $(u_1^t,\ldots, u_k^t) \in \overline{H}$ for $t>0$. Since $(u_{1,\mu},\ldots, u_{k,\mu})$ is a minimizer of $J_{\mu}$, we get 
\[
J_{\mu}(u_{1,\mu}, \ldots, u_{k,\mu}) \leq J_{\mu}(u_1^t,\ldots, u_k^t).
\]
Using the fact that $ \sum_{i=1}^k |{\Omega}_{u_{i,\mu}}| > a$, for $t>0$ sufficiently small, we obtain
\begin{align*}
\sum_{i=1}^k\displaystyle\int_{\Omega}|\nabla u_{i,\mu}|^2 + \mu\left[\sum_{i=1}^k|{\Omega}_{u_{i,\mu}}| - a\right]\leq  \displaystyle \sum_{i=1}^k\frac{\int_{\Omega} |\nabla (u_{i,\mu} - t)^{+}|^2}{\int_{\Omega} |(u_{i,\mu} - t)^{+}|^2} + \mu\left[\sum_{i=1}^k|{\Omega}_{u_i^t}| - a\right].
\end{align*}                                                                              
By using Lemma \ref{lemma:expansion_L^2}, we have

\begin{align*}
\sum_{i=1}^k\int_{\Omega}|\nabla u_{i,\mu}|^2 +  \mu \sum_{i=1}^k|\{0< u_{i,\mu} \leq t \}| \leq \sum_{i=1}^k \int_{\{ u_{i,\mu}> t \}} |\nabla u_{i,\mu}|^2 + 2t \int_{\Omega}u_{i,\mu} \int_{\Omega} |\nabla (u_{i,\mu} - t)^{+}|^2+ o(t).\\
\end{align*}
From the fact that $\|u_{i,\mu}\|_2 = 1$ for all $i=1, \ldots, k$, combined with the H\"older inequality, we infer that 
\begin{align*}
\displaystyle \sum_{i=1}^k\int_{\{0< u_{i,\mu}\leq t \}} |\nabla u_{i,\mu}|^2+ \mu \sum_{i=1}^k|\{0< u_{i,\mu} \leq t\}| \leq 
\displaystyle 2t \sum_{i=1}^k|\Omega_{u_{i,\mu}}|^{1/2}\int_{\Omega} |\nabla (u_{i,\mu} - t)^{+}|^2 + o(t).
\end{align*}

Now, we approximate $u_{i,\mu}$ by a mollifier and apply the coarea formula with Sard's Lemma to obtain the inequality (through the limit) to $u_{i,\mu}$ 
\begin{align*}
\displaystyle \sum_{i=1}^k\int_0^t \int_{\{u_{i,\mu} = s\}} \left( |\nabla u_{i,\mu}|+ \frac{\mu}{|\nabla u_{i,\mu}|}\right)d{\mathcal H}^{N-1}ds  \leq 
\displaystyle 2t \sum_{i=1}^k|\Omega_{u_{i,\mu}}|^{1/2}\int_{\Omega} |\nabla (u_{i,\mu} - t)^{+}|^2  + o(t).
\end{align*}

Minimizing the function $x \rightarrow x +\mu x^{-1}$ over the set  $\{x > 0\}$ leads to 

\begin{align*}
\displaystyle 2\sqrt{\mu}\sum_{i=1}^k\int_0^t \int_{\{u_{i,\mu} = s\}} d{\mathcal H}^{N-1}ds  \leq \displaystyle 2t \sum_{i=1}^k|\Omega_{u_{i,\mu}}|^{1/2}\int_{\Omega} |\nabla (u_{i,\mu} - t)^{+}|^2+ o(t).
\end{align*}
 
At this point, we can use the isoperimetric inequality, divide the equation by $t$, let $t \to 0$ to conclude that 
\begin{align*}
\displaystyle N | B_1|^{1/N}\sqrt{\mu} \sum_{i=1}^k|\Omega_{u_{i,\mu}}|^{\frac{N-1}{N}} \leq  \sum_{i=1}^k|\Omega_{u_{i,\mu}}|^{1/2}\int_{\Omega} |\nabla u_{i,\mu}|^2.
\end{align*}

Now, notice that
\[
\sum_{i=1}^k\int_{\Omega} |\nabla u_{i,\mu}|^2 \leq J_\mu(u_{1,\mu}, \ldots ,u_{k,\mu}) \leq J_\mu(\bar{u}_1, \ldots, \bar{u}_k) = J(\bar{u}_1, \ldots, \bar{u}_k) = \tilde{c}_a,
\] 
and $\tilde{c}_a$ does not depend on $\mu$. 
Hence,
\[
N | B_1|^{1/N}\sqrt{\mu}\left(\sum_{i=1}^k|\Omega_{u_{i,\mu}}|\right)^{\frac{N-1}{N}} \leq 2^{(k-1)/2} \, \tilde{c}_a\left(\sum_{i=1}^k|\Omega_{u_{i,\mu}}|\right)^{1/2}.
\] 
Therefore
\[
\sqrt{\mu} \leq \dfrac{2^{(k-1)/2}\,\tilde{c}_a}{N | B_1|^{1/N}}\left(\sum_{i=1}^k|\Omega_{u_{i,\mu}}|\right)^{\frac{2-N}{2N}} \leq \dfrac{2^{(k-1)/2}\, \tilde{c}_a}{N | B_1|^{1/N}}a^{\frac{2-N}{2N}},
\]  
which contradicts the assumption on the size of $\mu$. Hence, $\sum_{i=1}^k|{\Omega}_{u_{i,\mu}}|  \leq a$.

Finally, we conclude that
\[
J_\mu(u_{1,\mu}, \ldots, u_{k,\mu}) \leq J_\mu(\bar{u}_1, \ldots, \bar{u}_k) = \sum_{i=1}^k\int_{\Omega} |\nabla\bar{u}_i|^2 \leq J_\mu(u_{1,\mu},\ldots, u_{k,\mu}),
\]  
where the last inequality follows from \eqref{eq_01}.
\end{proof}

\smallbreak

\section{Continuity of minimizers}\label{sec:regularity}

Let $(\bar{u}_1, \ldots, \bar{u}_k) \in H_a$ be a minimizer to \eqref{main_func}. Up to replacing $\bar u_i$ by $|\bar u_i|$, we may suppose that $\bar u_i \geq 0$ in $\Omega$ for all $i=1, \ldots,k$. This is done throughout this paper. Then, inspired by \cite{ContiTerraciniVerziniAsymptotic,ContiTerraciniVerziniOPP,ContiTerraciniVerziniVariation}, set
\[
\hat{u}_i = \bar{u}_i - \sum_{j=1, j\neq i}^k \bar{u}_j, \quad \lambda_{\bar{u}_i} = \int_\Omega|\nabla\bar{u}_i|^2 \quad i=1,\ldots,k.
\]
We start by showing that minimizers of \eqref{main_func} are bounded functions.

\begin{Proposition}\label{prop_bound}
Let $(\bar{u}_1, \ldots, \bar{u}_k) \in H_a$ be a (nonnegative) minimizer to \eqref{main_func}. Then, for each $i=1, \ldots,k$, $\bar u_i$ satisfies
\begin{equation}\label{eq:subsolution}
-\Delta\bar{u}_i \leq \lambda_{\bar{u}_i}\bar{u}_i\;\;\;\mbox{ in } \;\; \Omega,
\end{equation}
in the sense of distributions. In particular, for each $i=1, \ldots,k$:
\begin{itemize}
\item $\bar{u}_i$ is a bounded function;
\item $\bar u_i$ is defined at every $x\in \Omega$, in the sense that each $x$ is a Lebesgue point.
\end{itemize}
\end{Proposition} 

\begin{proof} 
Let $\varphi \in C^{\infty}_c(\Omega)$ with $\varphi \geq 0$ and set $\check u_t$ as in \eqref{def-simples} for small $t$. Then, from Lemma \ref{def-simples}, $\check u_t \in H_a$. Using the fact that $(\bar{u}_1, \ldots, \bar{u}_k) \in H_a$ is a minimizer to \eqref{main_func} and \eqref{inverse-norm-lemma}, we infer that
\begin{align*}
\int_{\Omega}|\nabla\bar{u}_1|^2 &\leq  \dfrac{\int_{\Omega}|\nabla (\bar{u}_1 - t\varphi)^+|^2}{\|(\bar{u}_1 - t\varphi)^+\|_2^2} \leq  \dfrac{\int_{\Omega}|\nabla (\bar{u}_1 - t\varphi)|^2}{\|(\bar{u}_1 - t\varphi)^+\|_2^2} =  \int_\Omega |\nabla (\bar u_1-t\varphi)|^2\left( 1+2t\int_\Omega \bar u_1 \varphi + O(1)t^2\right) \\
& = \displaystyle\int_{\Omega}|\nabla\bar{u}_1|^2 -2t\int_{\Omega}\nabla\bar{u}_1\cdot \nabla\varphi +\displaystyle2t\int_{\Omega}|\nabla\bar{u}_1|^2\int_{\Omega}\bar{u}_1\varphi + O(1)t^2.
\end{align*}
Dividing the inequality above by $t$ and letting $t \to 0$, we obtain

\begin{equation}\label{eq_05}
\int_{\Omega}\nabla\bar{u}_1\cdot \nabla\varphi \leq \lambda_{\bar{u}_1}\int_{\Omega}\bar{u}_1\varphi,
\end{equation}
which implies \eqref{eq:subsolution}. By classical elliptic estimates, see for instance the proof of \cite[Lemma 4.2]{BHP2005}, we infer that $\bar{u}_1 \in L^{\infty}(\Omega)$. Similarly, we can also ensure that $\bar{u}_i \in L^{\infty}(\Omega)$ for all $i=2, \ldots, k$. Finally, the fact that every point is a Lebesgue point for each component $\bar u_i$ is a direct consequence of Proposition \ref{prop_Con}.
\end{proof}

The following result is crucial for everything that follows.

\begin{Proposition}\label{prop_03}
Let $(\bar{u}_1, \ldots, \bar{u}_k)$ be a minimizer of \eqref{main_func}. Then, for each $i=1, \ldots, k$, and for any nonnegative function $\varphi \in H_0^1(\Omega)$ such that $supp(\varphi) \subseteq B_{r}(x_0) \Subset \Omega$, 
\begin{equation}\label{inq1:Prop3.2}
\left\langle -\Delta\hat{u}_i -\lambda_{\bar{u}_i}\bar{u}_i+ \sum_{j\neq i}\lambda_{\bar{u}_j}\bar{u}_j, \varphi \right\rangle \geq -C\left(r^{N-1} + \|\varphi\|_{1}+ r\|\varphi\|_{2}^2 + r\|\nabla\varphi\|^2_2 + r\|\varphi\|_{1}\|\nabla\varphi\|^2_2 + r\|\varphi\|_{2}^2\|\nabla\varphi\|^2_2\right),
\end{equation}
where $C$ depends only on $\tilde{c}_a, N$, $\|\bar{u}_1\|_\infty, \ldots, \|\bar{u}_k\|_\infty$ and $a$. 
\end{Proposition} 
 
\begin{proof} 
Consider
\[
\tilde{u}_t=\left(\tilde{u}_{1,t}, \ldots, \tilde{u}_{k,t}\right)=\left(\frac{\left(\hat{u}_1+t \varphi\right)^{+}}{\left \|\left(\hat{u}_1+t \varphi\right)^{+}\right\|_2}, \frac{\left(\hat{u}_2-t \varphi \right)^{+}}{\left\|\left(\hat{u}_2-t \varphi\right)^+\right\|_2}, \ldots, \frac{\left(\hat{u}_k-t \varphi\right)^{+}}{\left\| \left(\hat{u}_k-t \varphi \right)^+\right\|_2}\right) ,
\]
with $t \in (0,1)$ sufficiently small. By using \eqref{eq_pen} and Lemma \ref{def-complexa}, we obtain
\begin{equation}\label{eq_24}
\sum_{i=1}^k\int_{\Omega}|\nabla\bar{u}_i|^2 \leq \displaystyle \dfrac{\int_{\Omega}|\nabla (\hat{u}_1 + t\varphi)^+|^2}{\|(\hat{u}_1 + t\varphi)^+\|^2_2} + \sum_{i=2}^k\dfrac{\int_{\Omega}|\nabla (\hat{u}_i - t\varphi)^+|^2}{\|(\hat{u}_i - t\varphi)^+\|^2_2} + \mu|\Omega_\varphi|.
\end{equation}
Since $\varphi \geq 0$, $0<t<1$, employing Lemma \ref{lemma:expansion_L^2}, we get
\[
\dfrac{1}{\|(\hat{u}_1 + t\varphi)^+\|^2_2}\int_{\Omega}|\nabla (\hat{u}_1 + t\varphi)^+|^2 \leq \displaystyle \int_{\Omega}|\nabla (\hat{u}_1 + t\varphi)^+|^2\left(1 - 2t\int_\Omega\bar{u}_1\varphi +O(1)\,  \|\varphi\|_2^2 \, t^2\right),
\]
and
\[
\dfrac{1}{\|(\hat{u}_i - t\varphi)^+\|^2_2}\int_{\Omega}|\nabla (\hat{u}_i - t\varphi)^+|^2 = \displaystyle \int_{\Omega}|\nabla (\hat{u}_i - t\varphi)^+|^2\left(1 + 2t\int_\Omega \bar{u}_i\varphi + O(1)\,  \|\varphi\|_2^2 \, t^2 \right),
\]
for all $i=2,\ldots,k$. In addition, since $0<t\leq 1$, we have that
\begin{align*}
\int_{\Omega}&|\nabla (\hat{u}_1 + t\varphi)^+|^2 \leq\int_{\Omega}|\nabla (\hat{u}_1 + t\varphi)|^2 \\
&= \sum_{j= 1}^k\int_\Omega |\nabla \bar{u}_j|^2 + 2t\int_\Omega \nabla \bar{u}_1\cdot \nabla \varphi- 2t\sum_{j=2}^k\int \nabla \bar{u}_j\cdot \nabla \varphi+t^2\int_\Omega |\nabla \varphi|^2\\
& \leq \sum_{j= 1}^k\int_\Omega |\nabla \bar{u}_j|^2 + t\int_\Omega|\nabla \bar{u}_1|^2 + t\int_{\Omega}|\nabla\varphi|^2 + t\sum_{j=2}^k\int_\Omega |\nabla \bar{u}_j|^2 + t\sum_{j=2}^k\int_\Omega|\nabla\varphi|^2 + t^2\int_\Omega |\nabla \varphi|^2 \\
& \leq \sum_{j= 1}^k\int_\Omega |\nabla \bar{u}_j|^2 + t\sum_{j= 1}^k\int_\Omega |\nabla \bar{u}_j|^2 +(k+1)t\int_{\Omega}|\nabla\varphi|^2,
\end{align*}
and, arguing similarly,
\begin{align*}
\int_{\Omega}&|\nabla (\hat{u}_i - t\varphi)^+|^2 \leq \sum_{j= 1}^k\int_\Omega |\nabla \bar{u}_j|^2 + t\sum_{j= 1}^k\int_\Omega |\nabla \bar{u}_j|^2 +(k+1)t\int_{\Omega}|\nabla\varphi|^2,
\end{align*}
for each $i = 2, \ldots, k$. Hence, since $\sum_{i=1}^k\int_{\Omega}|\nabla\bar{u}_i|^2$ is bounded and $\lambda_{\bar{u}_i} = \int_{\Omega} |\nabla \bar{u}_i|^2$,
\begin{align*}
\dfrac{1}{\|(\hat{u}_1 + t\varphi)^+\|^2_2}\int_{\Omega}|\nabla (\hat{u}_1 + t\varphi)^+|^2 \leq & \int_{\Omega}|\nabla (\hat{u}_1 + t\varphi)^+|^2 + O(1)t^2\|\varphi\|^2_{2}\|\nabla\varphi\|^2_2 + ct\|\varphi\|_{1}\\
& - 2t\lambda_{\bar{u}_1}\int_{\Omega}\bar{u}_1\varphi + ct^2\|\varphi\|_{1}\|\nabla\varphi\|^2_2 + O(1)t^2\|\varphi\|_{2}^2,
\end{align*}
and
\begin{align*}
\dfrac{1}{\|(\hat{u}_i - t\varphi)^+\|^2_2}\int_{\Omega}|\nabla (\hat{u}_i - t\varphi)^+|^2 \leq & \int_{\Omega}|\nabla (\hat{u}_i - t\varphi)^+|^2 + O(1)t^2\|\varphi\|^2_{2}\|\nabla\varphi\|^2_2 + ct\|\varphi\|_{1}\\
& + 2t\lambda_{\bar{u}_i}\int_{\Omega}\bar{u}_i\varphi + ct^2\|\varphi\|_{1}\|\nabla\varphi\|^2_2 + O(1)t^2\|\varphi\|_{2}^2.
\end{align*}
It follows from the fact $|\nabla(\hat{u}_1 + t \varphi)^+|^2 + \sum_{i=2}^k |\nabla(\hat{u}_i - t \varphi)^+|^2 = |\nabla(\hat{u}_1 + t \varphi)|^2$, combined with \eqref{eq_24}, Lemma \ref{def-complexa} and the estimates above that
\begin{align*}
0 \leq & \; 2t\int_\Omega \nabla\hat{u}_1\cdot \nabla\varphi + t^2\int_\Omega|\nabla\varphi|^2 - 2t\lambda_{\bar{u}_1}\int_{\Omega}\bar{u}_1\varphi + 2t\sum_{i=2}^k\lambda_{\bar{u}_i}\int_{\Omega}\bar{u}_i\varphi + ct^2\|\varphi\|_{1}\|\nabla\varphi\|^2_2  \\
& + ct\|\varphi\|_{1}  + O(1)t^2\|\varphi\|_{2}^2\|\nabla\varphi\|^2_2 + O(1)t^2\|\varphi\|_{2}^2 + \mu|\Omega_\varphi|.
\end{align*}
Hence, by dividing the inequality above by $2t$, we obtain
\begin{align*}
 -\int_\Omega \nabla\hat{u}_1\cdot \nabla\varphi + \lambda_{\bar{u}_1}\int_{\Omega}\bar{u}_1\varphi - \sum_{i=2}^k\lambda_{\bar{u}_i}\int_{\Omega}\bar{u}_i\varphi & \leq  ct\|\nabla\varphi\|_2^2 + ct\|\varphi\|_{1}\|\nabla\varphi\|^2_2 + c\|\varphi\|_{1}\\
 & \;\; + O(1)t\|\varphi\|_{2}^2\|\nabla\varphi\|^2_2 + O(1)t\|\varphi\|_{2}^2 + \mu\dfrac{|\Omega_\varphi|}{t}.   
\end{align*}
By choosing $t = r$ and using the fact that $|\Omega_\varphi|\leq |B_{r}|\leq c r^{N}$, we conclude that 
\begin{equation*}\label{eq_04} 
\left\langle \Delta\hat{u}_1 +\lambda_{\bar{u}_1}\bar{u}_1 - \sum_{i=2}^k\lambda_{\bar{u}_i}\bar{u}_i, \varphi \right\rangle \leq C\left(r^{N-1} + \|\varphi\|_{1}+ r\|\varphi\|_{2}^2 + r\|\nabla\varphi\|^2_2 + r\|\varphi\|_{1}\|\nabla\varphi\|^2_2 + r\|\varphi\|_{2}^2\|\nabla\varphi\|^2_2\right).  
\end{equation*}
Now, we do similar computations with the functions of the form 
\[
\frac{\left(\hat{u}_i+t \varphi\right)^{+}}{\left \|\left(\hat{u}_i+t \varphi\right)^{+}\right\|_2}, \frac{\left(\hat{u}_1-t \varphi \right)^{+}}{\left\|\left(\hat{u}_1-t \varphi\right)^+\right\|_2}, \ldots,\frac{\left(\hat{u}_{i-1}-t \varphi\right)^{+}}{\left\| \left(\hat{u}_{i-1}-t \varphi \right)^+\right\|_2}, \frac{\left(\hat{u}_{i+1}-t \varphi\right)^{+}}{\left\| \left(\hat{u}_{i+1}-t \varphi \right)^+\right\|_2}, \ldots, \frac{\left(\hat{u}_k-t \varphi\right)^{+}}{\left\| \left(\hat{u}_k-t \varphi \right)^+\right\|_2},
\]
to obtain
\[
\left\langle \Delta\hat{u}_i +\lambda_{\bar{u}_i}\bar{u}_i- \sum_{j\neq i}^k\lambda_{\bar{u}_j}\bar{u}_j, \varphi \right\rangle \leq C\left(r^{N-1} + r\|\varphi\|_{1}+ r\|\varphi\|_{2}^2 + r\|\nabla\varphi\|^2_2 + r\|\varphi\|_{1}\|\nabla\varphi\|^2_2 + r\|\varphi\|_{2}^2\|\nabla\varphi\|^2_2\right) , 
\]
for every $i=1, \ldots, k$.
\end{proof}

In what follows, we present a rescaled version of the previous proposition that will be useful later on.

\begin{Proposition}\label{p:rescaled}
Let $(\bar{u}_1, \ldots, \bar{u}_k)$ be a minimizer of \eqref{main_func}. Consider $R>0$, $x_0\in \Omega$ and a sequence $(r_n)_{n\in \mathbb{N}}$ such that $B_{r_nR}(x_0) \Subset \Omega$. For each $i=1, \ldots, k$, define
\[
u_{i,n}(x) = \bar{u}_i(x_0 + r_nx) .
\]
Then, given a nonnegative $\varphi \in H_0^1(\Omega)$ with $\mbox{supp}(\varphi) \subseteq B_{R}$, we have
\begin{align}\label{eq:32}
\left\langle -\Delta\hat{u}_{i, n} - r_n^2\lambda_{\bar{u}_i}\bar{u}_{i, n} + r_n^2\sum_{j\neq i}\lambda_{\bar{u}_j}\bar{u}_{j, n}, \varphi\right\rangle & \geq - Cr_nR\left(R^{N-2} + r_nR^{-1}\|\varphi\|_{1} + r_n^2\|\varphi\|^2_{2} +\right. \\
& \left.+\|\nabla\varphi\|^2_{2}+r_n^N\|\varphi\|_{1}\|\nabla\varphi\|^2_{2} +r_n^N\|\varphi\|^2_{2}\|\nabla\varphi\|^2_{2}\right), \nonumber
\end{align}
in the distributional sense, where the constant $C$ comes from Proposition \ref{prop_03} and does not depend on $n$.    
\end{Proposition}

\begin{proof}
Set $\varphi(x) = \tilde{\varphi}(x_0+r_nx)$, so that $supp(\tilde{\varphi}) \subset B_{r_nR}(x_0)$, and
\[
\|\tilde{\varphi}\|_{1} = r_n^N\|\varphi\|_{1}, \;\; 
\|\tilde{\varphi}\|^2_{2} = r_n^N\|\varphi\|^2_{2},  \;\;\|\nabla\tilde{\varphi}\|^2_{2} = r_n^{N-2}\|\nabla\varphi\|^2_{2}.
\]
We have
\begin{align*}
\int_{B_{R}}\nabla \hat{u}_{i,n}(x)\cdot\nabla\varphi(x) & = r_n^2\int_{B_{R}}\nabla \hat{u}_{i}(x_0+r_nx)\cdot\nabla\tilde{\varphi}(x_0+r_nx) \\
& = \dfrac{1}{r_n^{N-2}}\int_{B_{r_nR}(x_0)}\nabla \hat{u}_{i}(x)\cdot\nabla\tilde{\varphi}(x).
\end{align*}
Hence, from Proposition \ref{prop_03} we obtain
\begin{align*}
\int_{B_{R}}\nabla \hat{u}_{i,n}(x)\cdot\nabla\varphi(x) \geq & \; \dfrac{1}{r_n^{N-2}}\int_{B_{r_nR}(x_0)}\lambda_{\bar{u}_i}\bar{u}_i\tilde{\varphi} -\dfrac{1}{r_n^{N-2}}\sum_{j\neq i}\int_{B_{r_nR}(x_0)}\lambda_{\bar{u}_j}\bar{u}_j\tilde{\varphi}\\
& -\dfrac{C}{r_n^{N-2}}\left[(r_nR)^{N-1} + \|\tilde{\varphi}\|_{1} + r_nR\|\tilde{\varphi}\|^2_{2} + r_nR\|\nabla\tilde{\varphi}\|^2_{2}\right] \\
& - \dfrac{C}{r_n^{N-2}}\left[r_nR\|\tilde{\varphi}\|_{1} \|\nabla\tilde{\varphi}\|^2_{2} + r_nR\|\tilde{\varphi}\|^2_{2} \|\nabla\tilde{\varphi}\|^2_{2}\right] \\  
= & \; \dfrac{1}{r_n^{N-2}}\int_{B_{r_nR}(x_0)}\lambda_{\bar{u}_{i}}\bar{u}_{i,n} \left(\frac{x-x_0}{r_n}\right)\varphi\left(\frac{x-x_0}{r_n}\right) \\
& -\dfrac{1}{r_n^{N-2}}\sum_{j\neq i}\int_{B_{r_nR}(x_0)}\lambda_{\bar{u}_{j}}\bar{u}_{j,n} \left(\frac{x-x_0}{r_n}\right)\varphi\left(\frac{x-x_0}{r_n}\right) \\
& - Cr_nR^{N-1} - \dfrac{C}{r_n^{N-2}}\left(r_n^N\|\varphi\|_{1} + r_n^{N+1}R\|\varphi\|^2_{2} + r_n^{N-1}R\|\nabla\varphi\|^2_{2} \right)   \\
& - \dfrac{C}{r_n^{N-2}}\left(r_n^{2N-1}R\|\varphi\|_{1}\|\nabla\varphi\|^2_{2} + r_n^{2N-1}R\|\varphi\|^2_{2}\|\nabla\varphi\|^2_{2}\right) \\
= & \; r_n^2\int_{B_{R}}\lambda_{\bar{u}_i}\bar{u}_{i,n}\varphi - r_n^2\sum_{j\neq i}\int_{B_{R}}\lambda_{\bar{u}_j}\bar{u}_{j,n}\varphi - Cr_nR^{N-1} - Cr_n^2\|\varphi\|_{1} - Cr_n^3R\|\varphi\|^2_{2} \\
& -Cr_nR\|\nabla\varphi\|^2_{2} - Cr_n^{N+1}R\|\varphi\|_{1}\|\nabla\varphi\|^2_{2} - Cr_n^{N+1}R\|\varphi\|^2_{2}\|\nabla\varphi\|^2_{2}, 
\end{align*}
which implies \eqref{eq:32}.   
\end{proof}

The next proposition shows the intuitive fact that the information provided by deformations is stronger when the test functions $\varphi$ do not alter the measure of the positive sets.

\begin{Proposition}\label{prop_04}
Let $i \in \{1, \ldots, k\}$ and $B \subset \Omega$ be an open ball such that $|B\cap\{\hat{u}_i = 0\}| = 0$. Then
\begin{equation}\label{problinearineq}
-\Delta\hat{u}_i \geq \lambda_{\bar{u}_i}\bar{u}_i - \sum_{j=1, j\neq i}^k\lambda_{\bar{u}_j}\bar{u}_j \;\; \mbox{ in } \;\; B.
\end{equation}
\end{Proposition} 
 
\begin{proof}  
Without loss of generality, we prove \eqref{problinearineq} for $i= 1$. Let $\varphi \in C_c^\infty(B)$ be a nonnegative function and define, once again, the auxiliary deformations
\[
\tilde{u}_t=\left(\tilde{u}_{1,t}, \ldots, \tilde{u}_{k,t}\right)=\left(\frac{\left(\hat{u}_1+t \varphi\right)^{+}}{\left \|\left(\hat{u}_1+t \varphi\right)^{+}\right\|_2}, \frac{\left(\hat{u}_2-t \varphi \right)^{+}}{\left\|\left(\hat{u}_2-t \varphi\right)^+\right\|_2}, \ldots, \frac{\left(\hat{u}_k-t \varphi\right)^{+}}{\left\| \left(\hat{u}_k-t \varphi \right)^+\right\|_2}\right) ,
\]
with $t \in (0,1)$ sufficiently small. Unless stated, all the $L^2$ norms are taken in $\Omega$. From Lemma \ref{lemma:expansion_L^2}, we obtain
\begin{equation} \label{u_t-estimates}
\dfrac{1}{\|(\hat{u}_1 + t\varphi)^+\|_2^2} = 1 - 2t\int_B\bar{u}_1\varphi + o(t) \quad \text{ and } \quad \dfrac{1}{\|(\hat{u}_i - t\varphi)^+\|_2^2} = 1 + 2t\int_B\bar{u}_i\varphi + o(t).
\end{equation} 
By Lemma \ref{def-complexa}, it follows $u_{i,t} \cdot u_{j, t}\equiv 0$  for all $i \not= j$ and, by the assumption $|B\cap\{\hat{u}_i = 0\}| = 0$, we have $\sum_{i=1}^{k}|\Omega_{\bar {u}_i} \cap B|=|B|$. Therefore
\begin{align*}
\sum_{i=1}^{k}|\Omega_{u_{i,t}}|
&= \sum_{i=1}^{k}|\Omega_{u_{i,t}}\cap B | + {\sum}_{i=1}^{k}|\Omega_{u_{i,t}}\cap (\Omega \setminus B) | \leq |B| + \sum_{i=1}^{k}|\Omega_{\bar {u}_i}\cap (\Omega\setminus B)|\\
&= \sum_{i=1}^{k}|\Omega_{\bar{u}_i} \cap B|+  \sum_{i=1}^{k}|\Omega_{\bar{u}_i}\cap (\Omega\setminus B)| = \sum_{i=1}^{k}|\Omega_{\bar{u}_i}|\leq a.
\end{align*}

By combining \eqref{u_t-estimates}, the identity $ |\nabla(\hat{u}_1 + t \varphi)^+|^2 + \sum_{i=2}^{k}|\nabla(\hat{u}_i - t \varphi)^+|^2 = |\nabla(\hat{u}_1 + t \varphi)|^2$, the fact that $\varphi$ is supported in $B$ and $(\hat{u}_1 + t\varphi)^+=\bar{u}_1$ in $\Omega\setminus B$, $(\hat{u}_i - t\varphi)^+=\bar u_i$ in $\Omega \setminus B$ for $i>1$, we obtain
\begin{multline*}
J\left(\tilde{u}_{1,t}, \ldots, \tilde{u}_{k,t}\right)
= \int_{B}|\nabla(\hat{u}_1 + t \varphi)|^2 + \sum_{i=1}^{k}\int_{\Omega \setminus B} |\nabla \bar{u}_i|^2 - 2t \int_{B} \bar{u}_1\varphi \int_{B} |\nabla (\hat{u}_1 + t\varphi)^+|^2 \\
- 2t \int_{B} \bar{u}_1\varphi \int_{\Omega \setminus B} |\nabla \bar{u}_1|^2 +  2t \sum_{i=2}^{k}\int_{B} \bar{u}_i\varphi \int_{B} |\nabla (\hat{u}_i - t\varphi)^+|^2 + 2t  \sum_{i=2}^{k}\int_{B} \bar{u}_i\varphi \int_{\Omega \setminus B} |\nabla \bar{u}_i|^2 + o(t).
\end{multline*}
Since 
\[
\int_{B} |(\hat{u}_1 + t\varphi)^+|^2 \rightarrow \int_{B} |\nabla \bar{u}_1|^2\quad
 \text{ and } \quad \int_{B} |(\hat{u}_i - t\varphi)^+|^2 \rightarrow \int_{B} |\nabla \bar{u}_i|^2\: \mbox{ for all}  \: i>1
\]
as $t$ approaches to $0$, we deduce that
\begin{align*}
J\left(\tilde{u}_{1,t}, \ldots, \tilde{u}_{k,t}\right)
=& 2t \int_{B}\nabla\hat{u}_1 \cdot \nabla \varphi + \sum_{i=1}^{k}\int_{B} |\nabla \bar{u}_i|^2 + \sum_{i=1}^{k} \int_{\Omega\setminus B} |\nabla \bar{u}_i|^2 - 2t \int_{B} \bar{u}_1\varphi \int_{B} |\nabla \bar{u}_1|^2 \\
&- 2t \int_{B} \bar{u}_1\varphi \int_{\Omega \setminus B} |\nabla \bar{u}_1|^2 +  2t \sum_{i=2}^{k}\int_{B} \bar{u}_i\varphi \int_{B} |\nabla \bar{u}_i|^2 + 2t \sum_{i=2}^{k}\int_{B} \bar{u}_i\varphi \int_{\Omega \setminus B} |\nabla \bar{u}_i|^2 + o(t)\\
=& 2t \int_{B}\nabla\hat{u}_1\cdot \nabla \varphi+ \sum_{i=1}^{k}\int_{\Omega} |\nabla \bar{u}_i|^2 - 2t \int_{B} \bar{u}_1\varphi \int_{\Omega} |\nabla \bar{u}_1|^2 \\
&+  2t \sum_{i=2}^{k} \int_{B} \bar{u}_i\varphi \int_{\Omega}|\nabla \bar{u}_i|^2  + o(t).
\end{align*}
At this point, we use the fact that $(\bar{u}_1, \ldots, \bar{u}_k)$ is a minimizer of $J$; in addition, we denote $\lambda_{\bar {u}_i}:= \int_{\Omega}|\nabla \bar{u}_i|^2$. By passing to the limit as $t \to 0$, to get the following inequality
\[
\int_B \nabla\hat{u}_1\cdot \nabla\varphi - \lambda_{\bar{u}_1}\int_B\bar{u}_1\varphi + \sum_{i=2}^{k} \lambda_{\bar{u}_i}\int_B\bar{u}_i\varphi \geq 0,
\]
which finishes the proof.
\end{proof}

At this point, we introduce some auxiliary results. For $w\in H^1_0(\Omega)$, let $\displaystyle \lambda_w:=\int_\Omega |\nabla w|^2$.  For $(u_1,\ldots, u_k)\in H^1_0(\Omega; \R^k)$, define
\[
\hat u_i:= u_i-\sum_{j\neq i} u_j,\qquad i=1,\ldots, k.
\]
Then, as in \cite{ContiTerraciniVerziniAsymptotic,ContiTerraciniVerziniOPP,ContiTerraciniVerziniVariation}, given an open set $\mathcal{A}\subset \Omega$ and $\lambda_1,\ldots, \lambda_k>0$, set
\begin{align}\label{S-class}
\mathcal{S}_{\lambda_1,\ldots,\lambda_k}(\mathcal{A}):=&\Bigl\{ \left(w_1, \ldots, w_k\right) \in H^1(\mathcal{A}; \R^k): w_i \geq 0,\  w_i \cdot w_j=0 \text { if } i \neq j \text{ in } \Omega \Bigr. \\
&\Bigl.  \;\;\; -\Delta w_i \leq \lambda_i w_i ,\quad -\Delta \hat w_i \geq \lambda_{i}w_i-\sum_{j \neq i} \lambda_{j}w_j \text { in } \mathcal{A} \text { in the distributional sense}\Bigr\} \text {. }\nonumber
\end{align}

\begin{Lemma}\label{lemma:classeS} Let $\mathcal{A}\subset \Omega$ be an open set and $\lambda_1,\ldots, \lambda_k>0$. Take $(u_1,\ldots, u_k)\in \mathcal{S}_{\lambda_1,\ldots,\lambda_k}(\mathcal{A})\setminus \{(0,\ldots, 0)\}$. Then $u_i\in C^{0,1}_{loc}(\mathcal{A})$, and
\[
-\Delta u_i=\lambda_i u_i \text{ in the open set } \{u_i>0\}.
\]
In addition,
\[
|\{x\in \Omega:\ u_i(x)=0\ \text{ for } i=1,\ldots, k\}|=0.
\]
\end{Lemma}
\begin{proof}
The first conclusion follows from \cite[Theorem 8.3]{ContiTerraciniVerziniVariation}. The last sentence is a consequence of \cite[Corollary 8.5]{TavaresTerracini1}, taking therein $f_i(s):=\lambda_i s$.
\end{proof}

\begin{Remark}
For the problem without measure constraint \eqref{OPPclassical} (i.e., where the partition exhausts the whole $\Omega$), minimizers of the associated weak formulation belong to the class $\mathcal{S}_{\lambda_1,\ldots, \lambda_k}(\Omega)$ for some $\lambda_1,\ldots, \lambda_k>0$, see \cite[Lemma 2.1]{ContiTerraciniVerziniOPP}. Therefore, Proposition \ref{prop_04} shows that, in a region where the zero set has null measure, we are in the same situation, whereas Proposition \ref{prop_03} covers the  general case. The right hand side in \eqref{inq1:Prop3.2} can be seen as an error term, and in some sense allows to capture the transition from the positivity set $\{u_i>0\}$ to an empty region where $u_i\equiv 0$.
\end{Remark}

The following is a Liouville type result.
\begin{Lemma}\label{lemma:Liouville}
Let $(u_1,\ldots,u_k)\in H^1_{loc}(\R^N)\cap{L^\infty(\R^N)}$ nonnegative functions such that $u_i\cdot u_j\equiv 0$ for all $i\neq j$ and 
\[
-\Delta u_j\leq 0,\qquad -\Delta \hat{u}_j \geq 0 \qquad \text{ in the distributional sense in } \R^N, \quad \forall\, j.
\]
Then there exists $c\in \R$ and $i\in \{1,\ldots,k\}$ such that $u_i\equiv c$ and $u_j\equiv 0$ for $j\neq i$.
\end{Lemma}
\begin{proof}
First of all, observe that $(u_1,\ldots, u_k)\in \mathcal{S}_{(0,\ldots, 0)}(B_R(0))$ for every $R>0$. Then, by Lemma \ref{lemma:classeS}, each $u_i$ is a continuous function. By Proposition \ref{prop:Liouville} in the appendix, since all components are continuous, belong to $H^1_{loc}(\R^N)$, and  are bounded, we have that all components except possibly one is nontrivial. Without loss of generality, assume that $u_2\equiv \ldots \equiv u_k\equiv 0$. Then, from the assumptions,
\[
-\Delta u_1\leq 0,\qquad 0\leq -\Delta \hat{u}_1=-\Delta u_1,
\]
hence $u_1$ is harmonic and bounded in $\R^N$, thus it is constant.
\end{proof}

\smallbreak

We are ready to prove that minimizers of \eqref{main_func} are continuous functions. In particular, this shows that $\Omega_{\bar{u}_i}={\bar u_i>0}$, $i=1, \ldots, k$, are open sets.

\begin{Proposition}\label{prop_cont}
Let $U=(\bar{u}_1,\ldots \bar{u}_k)$ be a minimizer of \eqref{main_func}. Then each $\bar{u}_i$ is a continuous functions in $\Omega$.
\end{Proposition}
 
\begin{proof}  We recall that, by Proposition \ref{prop_bound}, each component $\bar u_{i}$ is defined at \emph{every} point. Given $x_0 \in \Omega$, we are going to prove the continuity of each $\bar u_{i}$ at $x_0$. Take a sequence $(x_n)_{n \in \mathbb{N}} \subset \Omega$ such that $x_n \to x_0$ and set $r_n := |x_0 - x_n| \to 0$. We split the proof into two cases:

\bigskip

\noindent {\it Case 1}: Suppose that, for some $n$, we have $|B_{r_n}(x_0)\cap\{U=0\}| = 0$. Then, from Propositions \ref{prop_bound} and \ref{prop_04}, 
\[
(\bar u_1,\ldots, \bar u_k)\in \mathcal{S}_{\lambda_{\bar u_1},\ldots, \lambda_{\bar u_k}}(B_{r_n}(x_0)).
\]
Then, by Lemma \ref{lemma:classeS}, we have $\bar u_i\in C^{0,1}_{loc}(B_{r_n}(x_0))$. In particular,  $\bar{u}_i$ is continuous at $x_0$.  

\bigskip

\noindent {\it Case 2}: Suppose that, for all $n$, $|B_{r_n}(x_0)\cap\{U=0\}| > 0$.  We introduce the auxiliary functions, for $i=1, \ldots, k$, 
\[
\bar{u}_{i,n}(x) = \bar{u}_i(x_0 +r_nx), \;\; \mbox{ with } \;\; x \in \R^N,
\]
where we are considering the extension of $\bar u_i$ by zero to $\R^N\setminus \Omega$. In particular, from Proposition \ref{prop_bound} and since $\bar u_i\in H^1_0(\Omega)$ is nonnegative in $\Omega$,
\begin{equation}\label{weak_solution}
- \Delta {\bar u}_{i, n}  \leq \lambda_{\bar{u}_i}r^2_n {\bar u}_{i, n} \quad \text{in} \quad \R^N
\end{equation}
in the distributional sense. Note that, since ${\bar u}_i$ is bounded, we have that ${\bar u}_{i, n}$ is uniformly bounded in $i$ and $n$.  Our aim is to show that $\bar u_{i,n}\to 0$ in $L^\infty_{loc}(\R^N)$, which proves the continuity of $\bar u_{i}$ (and shows that $\bar u_i(x_0)=0$). We split the proof of this in several steps.

Step 1. We show that, for each $i=1,\ldots, k$, there exist constants $c_1,\ldots, c_k\in \R$, where at most one is nonzero, such that
\begin{equation}\label{eq:Case2_Step1-a}
\bar{u}_{i,n} \rightharpoonup  c_i \quad \text{ weakly in $H_{loc}^1(\R^N)$, strongly in $L^2_{loc}(
\R^N)$ for each $i = 1, \ldots, k$.}
\end{equation}
Given $r< R' < R$, take $0\leq \varphi \in C^{\infty}(B_R)$   such that $\varphi \equiv 1 \: \text{in} \: B_r$ and $\varphi \equiv 0 \: \text{outside} \: B_{R'}$. From the definition of weak solution (with the test function  $ {\bar u}_{i, n} \varphi^2 \geq 0$) we get (since $\bar{u}_{i,n} \in H^1(\R^N)$)
\[
\int_{B_R} \nabla {\bar u}_{i, n} \cdot \nabla( {\bar u}_{i, n}\varphi^2) \leq \int_{B_R} \lambda_{\bar{u}_i}r^2_n  {{\bar u}_{i, n}}^2 \varphi^2 \leq C.
\]
Thus, 
\[ 
 \int_{B_R} |\nabla {\bar u}_{i, n}|^2 \varphi^2 \leq  C - 2 \int_{B_R}  {\bar u}_{i, n} \varphi\nabla   {\bar u}_{i, n} \cdot \nabla \varphi,
\]
which implies
\[ 
 \int_{B_R} |\nabla {\bar u}_{i, n}|^2 \varphi^2 \leq C + 2 \int_{B_R} {\bar u}_{i, n}^2 |\nabla\varphi|^2  + \frac{1}{2} \int_{B_R} |\nabla {\bar u}_{i, n}|^2 \varphi^2 \leq C + \frac{1}{2} \int_{B_R} |\nabla {\bar u}_{i, n}|^2\varphi^2
\]
and, therefore,
\[
\int_{B_r} |\nabla {\bar u}_{i, n}|^2 \leq C.
\]
From the bound above and the uniformly boundedness of $\bar{u}_{i,n}$, since $r$ is arbitrary there exists $\bar{u}_\infty=(\bar{u}_{1,\infty},\ldots, \bar{u}_{k,\infty})\in H^1_{loc}(\R^N)\cap L^\infty(\R^N)$ such that 
\[
\bar{u}_{i,n} \rightharpoonup  \bar{u}_{i, \infty} \quad \text{weakly in $H_{loc}^1(\R^N)$, strongly in $L^2_{loc}(
\R^N)$ for each $i = 1, \ldots, k$.}
\]

\smallbreak

Fix $R > 0$,  and let $n$ be large such that $B_{r_nR}(x_0) \Subset \Omega$.  Applying Proposition \ref{p:rescaled} to the functions $\bar{u}_{i,n}$, and by letting $n \to \infty$ in \eqref{eq:32}, we conclude that $\hat{u}_{i, \infty}$ solves
\[
-\Delta\hat{u}_{i, \infty} \geq 0 \;\;\mbox{ in }\;\;B_R.
\]
From the inequality above and \eqref{weak_solution} (by passing the limit as $r_n \to 0$), we can infer that $(\bar{u}_{1,\infty},\ldots, \bar{u}_{k,\infty})\in \mathcal{S}_{0,\ldots, 0}(B_R)$. Since $R > 0$ is arbitrary, we have that $(\bar{u}_{1,\infty},\ldots, \bar{u}_{k,\infty})$ satisfies the assumptions in Lemma \ref{lemma:Liouville}. Therefore, there exist constants $c_1, \ldots, c_k$ such that $\bar u_{i,\infty}\equiv c_i$,  at most one constant is nonzero and \eqref{eq:Case2_Step1-a} holds true.

Step 2. We now claim that 
\begin{equation}\label{eq:Case2_Step1}
\bar{u}_{i, n} \to c_i \quad \text{strongly in} \quad H^1_{loc}(B_R), \quad \forall \, \, i = 1, \ldots, k.
\end{equation}
In fact, by setting $\varphi$ as above and using $\bar{u}_{i, n} \varphi^2$ as a test function, we conclude that
\[
\int_{B_R} |\nabla \bar{u}_{i, n}|^2 \varphi^2 + 2 \int_{B_R} \bar{u}_{i, n} \varphi\nabla  \bar{u}_{i, n} \cdot \nabla \varphi \leq \int_{B_R} \lambda_{\bar u_i} r_n^2 \bar{u}_{i,n}^2 \varphi^2\to 0.
\]
On the other hand,
\begin{align*}
\int_{B_R} \bar{u}_{i, n} \varphi\nabla  \bar{u}_{i, n} \cdot \nabla \varphi  & = \int_{B_R}(\bar{u}_{i, n} - c_i)\varphi \nabla \bar{u}_{i, n}\cdot\nabla\varphi + \frac{1}{2}\int_{B_R}c_i \nabla \bar{u}_{i, n}\cdot\nabla(\varphi^2)\to 0
\end{align*}
by the weak convergence $\bar u_{i, n}\rightharpoonup c_i$ in $H^1(B_R)$ (which is strong in $L^2(B_R)$) and
\begin{align*}
\left|\int_{B_R}(\bar{u}_{i, n} - c_i)\varphi \nabla \bar{u}_{i, n}\cdot\nabla\varphi \right| \leq \left(\dfrac{1}{2}\|\varphi\|_\infty^2\|\nabla\varphi\|_2^2+\dfrac{1}{2}\|\nabla\bar{u}_{i, n}\|^2_{L^2(B_R)}\right)\|\bar{u}_{i, n} - \bar{u}_{i, \infty}\|_{L^2(B_R)}
\end{align*}
yields, by the definition of $\varphi$, to
\begin{equation}\label{eq_L2_limit}
    \int_{B_r} |\nabla \bar{u}_{i, n}|^2 \to 0, \quad \forall \,\, 0 < r < R,
\end{equation}
and the claim \eqref{eq:Case2_Step1} is proved.

Step 3. Suppose, without loss of generality, that $c_2=\ldots= c_k=0$. We show in this step that also $c_1=0$. In particular, $\bar u_{i,n}\to 0$ in $L^\infty_{loc}(\R^N)$ for every $i=1,\ldots, k$.

From the assumptions, we have $\bar{u}_{i,n} \to 0$ in $H^1_{loc}(B_R)$ for $i>1$ and, since $\bar{u}_{i,n}$ also satisfies \eqref{weak_solution},then  $\bar{u}_{i,n} \to 0$ in $L_{loc}^{\infty}(\R^N)$ by \cite[Theorem 8.17]{GiTr}. On the other hand, since $|B_{r_n}(x_0)\cap\{\hat{u}_1=0\}|>0$, we can take $y_n\in B_{r_n}(x_0)$ such that $\hat{u}_1(y_n)=0$. Write $y_n = x_0 + r_nz_n \in B_{r_n}(x_0)$, for some $z_n \in B_1$.

Now, for each large fixed $n$, take $r \leq  r_n$. Consider a test function $\varphi \in C_0^\infty(B_{2r}(y_n))$, such that $0 \leq \varphi \leq 1$ with $\varphi \equiv 1$ on $B_r(y_n)$ and $\|\nabla\varphi\|_{L^{\infty}(B_{2r}(y_n))} \leq C/r$. By using \eqref{eq:subsolution} once again, we see that each $\sigma_i := \Delta\bar{u}_i + \lambda_{\bar{u}_i}\bar{u}_i$ defines a positive measure. By Proposition \ref{prop_03} we infer that
\[
\sigma_1(B_r(y_n))  \leq  \langle\sigma_1, \varphi \rangle = \langle\sigma_1 - \sum_{i>1}\sigma_i, \varphi \rangle +\sum_{i>1}\langle\sigma_i, \varphi \rangle \leq  Cr^{N-1} + \sum_{i>1}\sigma_i(B_{2r}(y_n)).
\]
Therefore, since $\int_{B_r(y_n)}\lambda_{\bar{u}_i}\bar{u}_i\leq Cr^N$, for all $i=1, \ldots, k$ and $r\leq 1$,
\begin{align*}
\Delta\bar{u}_1(B_r(y_n)) & =  (\Delta\bar{u}_1 + \lambda_{\bar{u}_1}\bar{u}_1 - \lambda_{\bar{u}_1}\bar{u}_1)(B_r(y_n)) \\& \leq  Cr^{N-1} + \sum_{i>1}\sigma_i(B_{2r}(y_n))  \leq  C'r^{N-1} + \sum_{i>1}\Delta\bar{u}_i(B_{2r}(y_n)).
\end{align*}
By multiplying the inequality above by $r^{1-N}$ and integrating from $0$ to $r_n$, we obtain
\[
\int_{0}^{r_n}r^{1-N}\Delta\bar{u}_1(B_r(y_n))dr \leq Cr_n + \sum_{i>1}\int_0^{r_n}r^{1-N}\Delta\bar{u}_i(B_{2r}(y_n))dr.
\]

Now, we apply \eqref{eq_con} with $x_0 = y_n$ and $r = r_n$ to obtain (recall that $\hat{u}(y_n) = 0$)
\begin{align*}
C(N)\dashint_{\partial B_{r_n}(y_n)}\bar{u}_1 & = \int_{0}^{r_n}r^{1-N}\Delta\bar{u}_1(B_r(y_n))dr \leq Cr_n + \sum_{i>1}\int_0^{r_n}r^{1-N}\Delta\bar{u}_i(B_{2r}(y_n))dr \\
& \leq Cr_n + C(N)\sum_{i>1}\dashint_{\partial B_{2r_n}(y_n)}\bar{u}_i, 
\end{align*}
which leads to
\begin{equation}\label{mediaconv}
\dashint_{\partial B_1}\bar{u}_{1,n}(z_n + x) \leq Cr_n + C\sum_{i>1}\dashint_{\partial B_{2}}\bar{u}_{i,n}(z_n + x) \to 0,
\end{equation}
as $n$ goes to infinity (recall that $\bar{u}_{i,n} \to 0$ for all $i>1$). Up to a subsequence, we have $z_n \to z_\infty \in \overline{B}_1$ and $\bar{u}_{1,n}(z_n +x) \to c_1$ in $H^1(B_1)$ which implies strong convergence in $L^1(\partial B_1)$, and then
\[
\dashint_{\partial B_1}\bar{u}_{1,n}(z_n + x) \to c_1,
\]
and hence $c_1 = 0$, as wanted. The fact that also $\bar u_{1,n}\to 0$ in $L^\infty_{loc}(\R^N)$ is, again, a consequence of \cite[Theorem 8.17]{GiTr}.

\smallbreak

 Finally, from the the convergence $\bar{u}_{i,n} \to 0$ in $L^{\infty}_{loc}(\R^N)$, we obtain the continuity of each $\bar u_i$ at $x_0$, since $|\bar{u}_1(x)- \bar{u}_{1}(x_0)| \leq 2\|\bar{u}_{1,n}\|_{L^{\infty}(B_2)} \to 0$ for all $x\in B_{2r_n}(x_0)$ (and also $U(x_0)=0)$).\qedhere
\end{proof} 
 
\section{Lipschitz regularity of minimizers}\label{secLip}

Let $U := (\bar{u}_1, \ldots, \bar{u}_k) \in H_a$ be a minimizer of \eqref{main_func}, extended by zero in $\R^N\setminus \Omega$. Now, we introduce two quantities related to whether a point $x$ belongs or not to $\partial\Omega_{\bar{u}_i}$, for some $i=1, \ldots, k$. Define the multiplicity of a point $x\in \Omega$ as being
\[
m(x) := \#\{i \; ;\; |\Omega_{\bar{u}_i}\cap B_r(x)| > 0, \mbox{ for all } r > 0\},
\]
and 
\[
Z_{\ell}(U) = \{x \in \Omega \; ;\; m(x) \geq \ell\}.
\]

Consider the function $\Sigma: \bar{\Omega} \times (0, \infty) \to \mathbb{R}$ defined as
\[
\Sigma(x,r) := \dfrac{1}{r^N}\int_{B_r(x)}|\nabla U|^2, \quad \text{for} \quad (x,r) \in \bar{\Omega}\times (0,\infty).
\]
In order to prove the interior local Lipschitz regularity, it is enough to show that $\Sigma$ is bounded over $\Omega'\times(0,\infty)$, for every $\Omega'$ compactly contained in $\Omega$. So, fix such a set  $\Omega'$  and suppose, by contradiction, that $\Sigma$ is unbounded in $\Omega'\times (0, \infty)$. Then, there exist sequences $(x_n)_{n \in \mathbb{N}} \subset \Omega'$ and $r_n \to 0$ such that $B_{r_n}(x_n) \subset \Omega$ and
\begin{equation}\label{eq_cont}
 \lim_{n\to\infty}\dfrac{1}{r_n^N}\int_{B_{r_n}(x_n)}|\nabla U|^2 = +\infty.   
\end{equation}

In what follows, we present two technical lemmas that will be applied mainly to the sets $Z_\ell$. Their proofs can be found in \cite{ContiTerraciniVerziniVariation} but, for completeness, we also include them here. 

\begin{Lemma}\label{lemma:bbound}
Let $(x_n, r_n)$ be as in \eqref{eq_cont}. Then, there exists a sequence $r'_n \to 0$ as $n\to\infty$, with $(x_n, r_n')$ satisfying \eqref{eq_cont} and such that
\begin{equation}\label{eq_bbound}
\int_{\partial B_{r'_n}(x_n)}|\nabla U|^2 \leq \dfrac{N}{r'_n}\int_{B_{r'_n}(x_n)}|\nabla U|^2 \quad \text{for all} \quad n \in \mathbb{N}. 
\end{equation}
\end{Lemma}

\begin{proof}
First, notice that
\begin{equation}\label{eq_27}
\dfrac{d}{dr}\left(\dfrac{1}{r^N}\int_{B_{r}(x_n)}|\nabla U|^2 \right) =  \dfrac{1}{r^N}\left(\int_{\partial B_{r}(x_n)}|\nabla U|^2 -  \dfrac{N}{r}\int_{B_{r}(x_n)}|\nabla U|^2\right) =: \dfrac{1}{r^N}f(x_n,r).
\end{equation}
Hence, it is enough to find a sequence $(r'_n)_{n \in \mathbb{N}}$ such that \eqref{eq_cont} and $f(x_n, r'_n) \leq 0$ holds true. Define $r'_n := \inf\{r \geq r_n : f(x_n,r) \leq 0\}$. Since $\Sigma(x_n, r) \to 0$ as $r \to \infty$ and $\Sigma(x_n, r_n)>0$, we infer that $\frac{d}{dr}\Sigma(x_n,r)\leq 0$ for some $r$ sufficiently large, hence $r_n'< \infty$ for all $n$. Moreover, since $\Sigma(x_n, r_n') \geq \Sigma(x_n, r_n) \to \infty$, we then infer that $r_n'\to 0$. Then, up to a finite number of indices, $B_{r_n'}(x_n) \subset \Omega$. Moreover, from the definition of $r'_n$, $f(x_n, r_n')\leq0$, that is, \eqref{eq_bbound} is verified. This finishes the proof.
\end{proof}

From the last lemma, since $\Sigma$ is not bounded over $\Omega'\times(0,\infty)$, then from now on we may assume the existence of sequences $(x_n)_{n \in \mathbb{N}} \subset \Omega'$, $r_n \to 0$ satisfying \eqref{eq_cont} and \eqref{eq_bbound}. 

\begin{Lemma}\label{lem_dist}
Let $A \subset \bar{\Omega}$ be such that $
\mbox{dist}(x_n, A) \leq C \, r_n,$ for all $n$, and assume that \eqref{eq_cont} holds true. Then, there exist sequences $(x'_n)_{n \in \mathbb{N}}\subset A$ and $r'_n\to 0$ such that $B_{r_n'}(x_n') \subset \Omega$ and $(x_n', r_n')$ satisfies \eqref{eq_cont}  and \eqref{eq_bbound}.
\end{Lemma}

\begin{proof}
 By assumption, we can find $x'_n \in A$ such that $\mbox{dist}(x_n, x'_n) \leq 2Cr_n$. Now, set $r'_n :=(2C + 1)r_n$ and observe that $B_{r_n}(x_n) \subset B_{r'_n}(x'_n)$ and, for all $n$ sufficiently large, $B_{r_n'}(x'_n) \subset \Omega$. Hence
 \[
 \dfrac{(2C+1)^{-N}}{r_n^N}\int_{B_{r_n}(x_n)}|\nabla U|^2 \leq \dfrac{1}{(r'_n)^N}\int_{B_{r'_n}(x'_n)}|\nabla U|^2 = \Sigma(x_n', r_n'). 
 \]
 Since the left-hand side in the inequality above goes to infinity, the same holds for the right-hand side, hence \eqref{eq_cont} is satisfied. By eventually changing the radii $r_n'$ (recall Lemma \ref{lemma:bbound}) we may assume without loss of generality that  \eqref{eq_bbound} is also true.
\end{proof}

\begin{Remark}\label{rem_dist}
Notice that, for $\ell \geq 0$, if $m(x_n) = \ell$ and $\mbox{dist}(x_n, Z_{\ell+1}) < r_n$ for all $n$, then by Lemma \ref{lem_dist}, we can find sequences $(x'_n)_{n \in \mathbb{N}}\subset Z_{\ell+1}$ and $r'_n\to 0$  such that $B_{r_n'}(x_n') \subset \Omega$ for all $n$, with $(x_n', r_n')$ satisfying \eqref{eq_cont} and \eqref{eq_bbound}. In particular, $m(x'_n) \geq \ell + 1$ for all $n$ and $\Omega' \bigcup_{n \in \mathbb{N}}B_{r_n'}(x_n')\bigcup_{n \in \mathbb{N}}B_{r_n}(x_n)$ is compactly contained in $\Omega$.
\end{Remark}

Now, we aim to get a contradiction from with \eqref{eq_cont}. We achieve this by splitting the proof into cases, based on the quantity $m(x_n)$. We first treat the case $m(x_n) \geq 2 $.

\begin{Proposition}\label{prop:m2}
 Under the conditions above, suppose that, up to a subsequence, $m(x_n) \geq 2$ for all n. Then \eqref{eq_cont} cannot hold true.   
\end{Proposition}

\begin{proof}
It follows from hypothesis $m(x_n)\geq 2$, combined with the continuity of $\bar{u}_i$, that $\bar{u}_i(x_n) = 0$, for every $i=1, \ldots, k$, and for all $n$. Set $v_i:=\bar{u}_i$ and $w_i:=\sum_{j\neq i}^k\bar{u}_j$. Applying Proposition \ref{prop_bound}, we have that
\[
-\Delta{v_i} \leq \gamma \;\;\mbox{ in }\;B_{r_n}(x_n) \:\:\mbox{ and }  \:\: -\Delta{w_i} \leq \gamma \;\;\mbox{ in }\; B_{r_n}(x_n),
\]
in the sense of distributions, where 
\[
\gamma := \max\left\{\lambda_{\bar{u}_i}\|\bar{u}_i\|_{L^\infty(\Omega)}, \sum_{j\neq i}^k\lambda_{\bar{u}_j}\|\bar{u}_j\|_{L^\infty(\Omega)}\right\}.
\]
Now, we employ Lemma \ref{jck_lemma} in $B_{r_n}(x_n) \subset \Omega$ to the pair $(v_i, w_i)$ to conclude that  
\[
\left(\dfrac{1}{r_n^2}\int_{B_{r_n}(x_n)}\dfrac{|\nabla v_i|^2}{|x-x_n|^{N-2}}\right)\left(\dfrac{1}{r_n^2}\int_{B_{r_n}(x_n)}\dfrac{|\nabla w_i|^2}{|x-x_n|^{N-2}}\right) \leq C,
\]
where $C > 0$ is independent of $n$. Since $|x-x_n| \leq r_n$ for $x\in B_{r_n}(x_n)$, we obtain
\[
\left(\dfrac{1}{r_n^N}\int_{B_{r_n}(x_n)}|\nabla v_i|^2\right)\left(\dfrac{1}{r_n^N}\int_{B_{r_n}(x_n)}|\nabla w_i|^2\right) \leq C.
\]
From the inequality above and \eqref{eq_cont}, we infer the existence of only one component, say $\bar{u}_1$, such that
\begin{equation}\label{eq_29}
\dfrac{1}{r_n^N}\int_{B_{r_n}(x_n)}|\nabla \bar{u}_1|^2 \to \infty,    
\end{equation}
and 
\begin{equation}\label{eq_30}
\dfrac{1}{r_n^N}\int_{B_{r_n}(x_n)}|\nabla \bar{u}_i|^2 \to 0,     
\end{equation}
for every $i=2, \ldots, k$. For the sake of clarity, we split the proof into five steps.
\bigskip

\noindent\textit{Step 1 (Blow up sequence):} We introduce the blow-up sequence
\[
U_n(x) = \dfrac{1}{L_nr_n}U(x_n+r_nx), \;\;\mbox{ for }\; x\in B_1,
\]
where 
\[
L_n^2 := \dfrac{1}{r_n^N}\int_{B_{r_n}(x_n)}|\nabla U|^2\to \infty.
\]
Let us denote $U_n = (\bar{u}_{1,n}, \ldots, \bar{u}_{k,n})$. As a consequence of the definition of $L_n$, 
\begin{align} \label{eq:44}
\int_{B_1}|\nabla U_n|^2 = \dfrac{1}{L_n^2}\int_{B_1}|\nabla U(x_n + r_nx)|^2 = \dfrac{1}{L_n^2r_n^N}\int_{B_{r_n}(x_n)}|\nabla U|^2 = 1.     
\end{align}
Hence, from \eqref{eq_bbound}, we conclude that $\int_{\partial B_1}|\nabla U_n|^2$ is also bounded. 
\smallbreak
From here, we split the proof into two cases:
\smallbreak
\noindent\textbf{Case 1:} Suppose that there exists a positive constant $C$, independent of $n$ such that $\|U_n\|_{L^2(B_1)} \leq C$. In this case, using \eqref{eq:44}, there exists $U_\infty \in H^1(B_1)$ such that $U_n \rightharpoonup U_\infty$ in $H^1(B_1)$ and then we can proceed to {\it Step 2}.

\bigskip 

\noindent\textit{Step 2 ($U_\infty \not\equiv 0$):}  We denote $U_\infty = (\bar{u}_{1,\infty}, \ldots, \bar{u}_{k,\infty})$. From \eqref{eq_30}, we get that 
\[
\int_{B_1}|\nabla\bar{u}_{i, n}|^2 \to 0, \;\;\mbox{ for every }\; i = 2, \ldots, k,
\]
which implies $\|\nabla\bar{u}_{i, \infty}\|_{L^2(B_1)} = 0$ for every $i = 2, \ldots, k$. Since $\bar{u}_{i, n}\cdot \bar{u}_{j,n} = 0$ for $i \neq j$, the a.e. convergence implies that  at most  one component of $U_\infty$ is nonzero. In view of the definition of $U_n$ and Proposition \ref{prop_bound}, we obtain
\[
-\Delta\bar{u}_{1, n}(x) = \dfrac{r_n}{L_n}(-\Delta \bar{u}_{1}(x_n + r_nx)) \leq \dfrac{r_n}{L_n}\lambda_{\bar{u}_1}\bar{u}_1(x_n+r_nx) = r_n^2\lambda_{\bar{u}_1}\bar{u}_{1, n}(x), \quad \text{for $x \in B_1$.}
\]
 Multiplying the inequality above by $\bar{u}_{1,n}$ and integrating by parts, we conclude
\[
\int_{B_1}|\nabla\bar{u}_{1, n}|^2 \leq \int_{\partial B_1}\bar{u}_{1, n}\dfrac{\partial\bar{u}_{1, n}}{\partial\nu} + r_n^2\lambda_{\bar{u}_{1}}\int_{B_1}\bar{u}^2_{1, n},
\]
where $\nu$ is an outward unit normal vector. Now, suppose that $\bar{u}_{1,\infty} \equiv 0$. Employing the compacts embeddings of $H^1(B_1)$ in $L^2(B_1)$ and in $L^2(\partial B_1)$, and the fact that $\|\nabla U_n\|_{L^2(\partial B_1)}$ is bounded, we obtain that 
\[
\int_{\partial B_1}\bar{u}_{1, n}\dfrac{\partial\bar{u}_{1, n}}{\partial\nu} + r_n^2\lambda_{\bar{u}_{1}}\int_{B_1}\bar{u}^2_{1, n} \to 0,
\]
which implies $\|\nabla \bar u_{1,n}\|_{L^2(B_1)} \to 0$, and so $\|\nabla U_n\|_{L^2(B_1)} \to 0$, a contradiction with $\|\nabla U_n\|_{L^2(B_1)} = 1$. Therefore, $\bar{u}_{1, \infty} \not\equiv 0$, which implies $U_\infty =(\bar{u}_{1, \infty}, 0, \ldots, 0)\not\equiv 0$. 
\bigskip

\noindent\textit{Step 3 ($\bar{u}_{1, \infty}$ is a harmonic function):} First, we recall that, for every $i=1, \ldots, k$,
\[
-\Delta\bar{u}_{i, n}(x) \leq r_n^2\lambda_{\bar{u}_i}\bar{u}_{i, n}(x) \;\;\mbox{ in }\; B_1.
\]
Hence, by passing the limit as $n \to \infty$, we see in particular that
\begin{equation}\label{eq_31}
-\Delta\bar{u}_{1, \infty}(x) \leq 0 \;\;\mbox{ in }\; B_1. 
\end{equation}
Now, reasoning as Proposition \ref{p:rescaled} we can infer that for any nonnegative $\varphi \in H_0^1(B_1)$ with $\mbox{supp}(\varphi) \subseteq B_{2s}$, $2s \in (0,1]$, we have
\begin{align}\label{eq_32}
\left\langle -\Delta\hat{u}_{i, n} - r_n^2\lambda_{\bar{u}_i}\bar{u}_{i, n} + r_n^2\sum_{j\neq i}\lambda_{\bar{u}_j}\bar{u}_{j, n}, \varphi\right\rangle & \geq -\dfrac{Cs}{L_n}\left(s^{N-2} + r_ns^{-1}\|\varphi\|_{1} + r_n^2\|\varphi\|^2_{2} +\right. \\
& \left.+\|\nabla\varphi\|^2_{2}+r_n^N\|\varphi\|_{1}\|\nabla\varphi\|^2_{2} +r_n^N\|\varphi\|^2_{2}\|\nabla\varphi\|^2_{2}\right), \nonumber
\end{align}
in the sense of distributions. Now, taking $\varphi \in C^\infty_c(B_1)$, and recalling that $r_n \to 0$, $L_n \to \infty$ and $\bar{u}_{j,n} \to 0$ in $H^1(B_1)$ for every $j = 2,\ldots, k$, we conclude
\[
-\Delta\bar{u}_{1, \infty} = -\Delta\hat{u}_{1, \infty} \geq 0 \;\;\mbox{ in }\; B_1,
\]
and combined with \eqref{eq_31}, we infer that 
\[
\Delta\bar{u}_{1, \infty} = 0 \;\;\mbox{ in }\;\;B_1.
\] 
According to the maximum principle, we have $\bar{u}_{1, \infty} > 0$ in $B_1$ (recall that $\bar{u}_{1, \infty} \geq 0$ and $\bar{u}_{1, \infty} \not\equiv 0$). 
\bigskip

\noindent\textit{Step 4 (Contradiction):} At this point, we apply Proposition \ref{prop_Con} to the functions $\bar{u}_{i, n}$. Denote $\sigma_{i,n} := \Delta\bar{u}_{i,n} + 
r_n^2\lambda_{\bar{u}_i}\bar{u}_{i,n}$, and recall that $\sigma_{i,n} \geq 0$, for every $i =1, 2, \ldots, k$. Notice that, for $n$ sufficiently large, we have $\Delta\bar{u}_{i,n} \geq -1$. Moreover, by Proposition \ref{prop_bound}, we have $\bar{u}_{i,n} \in L^{\infty}(B_{1/2})$. Therefore, by applying Proposition \ref{prop_Con}, we obtain, for all $r \in (0, 1/2)$, that
\begin{align*}
 \dashint_{\partial B_r}[\bar{u}_{1,n} - \bar{u}_{1,n}(0)] & = C(N)\int_0^rs^{1-N}\left[\int_{B_s}d(\Delta\bar{u}_{1,n})\right]ds = C(N)\int_0^rs^{1-N}\Delta\bar{u}_{1,n}(B_s)ds \\
 & = C(N)\int_0^rs^{1-N}\sigma_{1,n}(B_s)ds - C(N)\int_0^rs^{1-N}(r_n^2\lambda_{\bar{u}_{1}}\bar{u}_{1,n})(B_s)ds \\
 & \leq C(N)\int_0^rs^{1-N}\sigma_{1,n}(B_s)ds.
\end{align*}
Now, for each fixed $s \in (0,r)$, we take a test function $\varphi \in C_0^\infty(B_{2s})$ such that $0 \leq \varphi \leq 1$ with $\varphi \equiv 1$ in $B_s$ and $\|\nabla\varphi\|_{L^\infty(B_{2s})} \leq C/s$. Hence, 
\begin{align*}
 \sigma_{1,n}(B_s) & \leq \left\langle \sigma_{1,n}, \varphi\right\rangle  = \left\langle \sigma_{1,n} - \sum_{i=2}^k\sigma_{i,n}, \varphi\right\rangle  + \sum_{i=2}^k\left\langle \sigma_{i,n}, \varphi\right\rangle \\
 & \leq \left\langle \Delta\hat{u}_{1,n} + r_n^2\lambda_{\bar{u}_{1}}\bar{u}_{1,n} - r_n^2\sum_{i=2}^k\lambda_{\bar{u}_{i}}\bar{u}_{i,n}, \varphi\right\rangle  +\sum_{i=2}^k\sigma_{i,n}(B_{2s})\leq \dfrac{Cs^{N-1}}{L_n} +\sum_{i=2}^k\sigma_{i,n}(B_{2s}),
\end{align*}
where in the last inequality we applied \eqref{eq_32} and used the fact that $r_n\leq 1$ and $|B_{2s}| \leq C(N)s^{N}$ (and hence we have $\|\varphi\|_{1} \leq C(N)s^N$, $\|\varphi\|^2_{2} \leq C(N)s^N$  and $\|\nabla\varphi\|^2_{2} \leq C(N)s^{N -2}$). Therefore,
\[
s^{1-N}\sigma_{1,n}(B_s) \leq \dfrac{C}{L_n} + s^{1-N}\sum_{i=2}^k\sigma_{i,n}(B_{2s}), \quad \text{for all $s \in (0,r)$.}
\]
Plugging the inequalities above, we get (recall that $\bar{u}_{1,n}(0) = 0$, as observed in the first line of the proof)
\begin{equation} \label{eq_39}
  \dashint_{\partial B_r}\bar{u}_{1,n} \leq C(N) \int_0^r\left[ \dfrac{C}{L_n} + s^{1-N}\sum_{i=2}^k\sigma_{i,n}(B_{2s})\right]ds. 
\end{equation}
Now, we use Proposition \ref{prop_Con} to estimate the last term in \eqref{eq_39}:
\begin{align*}
 \int_0^rs^{1-N}\sum_{i=2}^k\sigma_{i,n}(B_{2s})ds & = \sum_{i=2}^k \int_0^rs^{1-N}(\Delta\bar{u}_{i, n} + r_n^2\lambda_{\bar{u}_{i}}\bar{u}_{i,n})(B_{2s})ds \\ & = \sum_{i=2}^k \int_0^rs^{1-N}\Delta\bar{u}_{i,n}(B_{2s})ds + \sum_{i=2}^k r_n^2\lambda_{\bar{u}_{i}}\int_0^rs^{1-N}\int_{B_{2s}}\bar{u}_{i,n}\, dx ds \\
 & \leq C(N)\sum_{i=2}^k\dashint_{\partial B_r}\bar{u}_{i,n} + \sum_{i=2}^k \int_0^r\frac{C(N)sr_n\lambda_{\bar{u}_{i}} \|\bar{u}_{i}\|_{L^\infty(B_1)}}{L_n}ds \\
 & = C(N)\sum_{i=2}^k\dashint_{\partial B_r}\bar{u}_{i,n} + \sum_{i=2}^k \frac{C(N)r^2r_n\lambda_{\bar{u}_{i}} \|\bar{u}_{i}\|_{L^\infty(B_1)}}{2 L_n}.
\end{align*}
Plugging the inequality above into \eqref{eq_39} and multiplying it by $|\partial B_r|$ yields to 
\begin{align*}
\int_{\partial B_r}\bar{u}_{1,n} & \leq \dfrac{C(N)r^{N}}{L_n} + C(N)\sum_{i=2}^k\int_{\partial B_r}\bar{u}_{i,n} +  \sum_{i=2}^k \frac{C(N)r_n\lambda_{\bar{u}_{i}} \|\bar{u}_{i}\|_{L^\infty(B_1)}r^{N+1}}{L_n}, \quad \text{for all $r \in (0,1/2)$.}
\end{align*}
Now, we integrate the inequality above with respect to $r$ to obtain
\[
\int_{B_{1/2}}\bar{u}_{1,n} \leq \dfrac{C(N)}{L_n} + C(N)\sum_{i=2}^k\int_{B_{1/2}}\bar{u}_{i,n} + \sum_{i=2}^k \frac{C(N)r_n\lambda_{\bar{u}_{i}} \|\bar{u}_{i}\|_{L^\infty(B_1)}}{L_n}. \]
Finally, we pass the limit as $n \to 0$, and recalling that $L_n \to \infty$, $r_n \to 0$, $\bar{u}_{1,n} \to \bar{u}_{1,\infty} > 0$ and $\|\bar{u}_{i,n}\|_{L^1(B_{1/2})} \to 0$, we get
\[
  0 < \int_{B_{1/2}}\bar{u}_{1, \infty}  \leq 0,
\]
which is a contradiction.

\noindent\textbf{Case 2:} Assume now that $\|U_n\|_{L^2(B_1)} \rightarrow \infty$, and set $V_n := U_n\cdot\|U_n\|_{L^2(B_1)}^{-1} = (v_{1,n}, \ldots, v_{k,n})$. Hence, 
\[
\|V_n\|_{L^2(B_1)} = 1 \;\;\;\mbox{ and }\;\; \|\nabla V_n\|_{L^2(B_1)} \to 0.
\]
Therefore, there exists $V_\infty \in H^1(B_1)$ such that $V_n \rightharpoonup V_\infty = (v_{1,\infty}, \ldots, v_{k,\infty})$ locally in $H^1(B_1)$ and $\|V_\infty\|_{L^2(B_1)} = 1$. Since, for every $\varphi\in H^1(B_1)$,
\[
\int_{B_1} \nabla v_{i,n}\cdot \nabla \varphi\to 0 \quad \: \text{and} \quad \: \int \nabla v_{i,n}\cdot \nabla \varphi\to \int \nabla v_{i,\infty} \cdot \nabla \varphi,
\]
for every $i=1, \ldots, k$, we have that $\|\nabla V_\infty\|_{L^2(B_1)} = 0$, consequently $V_\infty = (c_1, \ldots, c_k)$, where $0 \leq c_i \in \mathbb{R}$, for every $i = 1, \ldots, k$. Moreover, since $v_{i, n}\cdot v_{j,n} = 0$ for $i \neq j$, the a.e. convergence implies that only one component of $V_\infty$ is nonzero;  without loss of generality, we can say $V_\infty = (c_1, 0, \ldots, 0)$. In particular, we have $v_{1, \infty} > 0$, and $\Delta v_{1, \infty} = 0$. At this point, we can apply {\it Step 5} to $V_n$.
\medbreak

\noindent\textit{Step 5 (Contradiction for $V_n$):} Following the general lines of {\it Step 4}, we define $\tilde{\sigma} := \Delta v_{i,n} + r_n^2\lambda_{\bar{u}_i}v_{i, n} \geq 0$, and apply Proposition \ref{prop_Con} to $v_n$ and we obtain that for all $r \in (0, 1/2)$,
\begin{align*}
 \dashint_{\partial B_r}[v_{1,n} - v_{1,n}(0)] & = C(N)\int_0^rs^{1-N}\left[\int_{B_s}d(\Delta v_{1,n})\right]ds = C(N)\int_0^rs^{1-N}\Delta\bar{u}_{1,n}(B_s)ds \\
 & \leq C(N)\int_0^rs^{1-N}\tilde{\sigma}_{1,n}(B_s)ds.
\end{align*}
Once again, for each fixed $s \in (0,r)$, we take a test function $\varphi \in C_0^\infty(B_{2s})$ such that $0 \leq \varphi \leq 1$ with $\varphi \equiv 1$ in $B_s$ and $\|\nabla\varphi\|_{L^\infty(B_{2s})} \leq C/s$. Hence, 
\begin{align*}
 \tilde{\sigma}_{1,n}(B_s) & \leq \left\langle \tilde{\sigma}_{1,n}, \varphi\right\rangle  = \left\langle \tilde{\sigma}_{1,n} - \sum_{i=2}^k\tilde{\sigma}_{i,n}, \varphi\right\rangle  + \sum_{i=2}^k\left\langle \tilde{\sigma}_{i,n}, \varphi\right\rangle \\
 & \leq \left\langle \Delta v_{1,n} + r_n^2\lambda_{\bar{u}_{1}} v_{1,n} - r_n^2\sum_{i=2}^k\lambda_{\bar{u}_{i}} v_{i,n}, \varphi\right\rangle  +\sum_{i=2}^k\tilde{\sigma}_{i,n}(B_{2s}) \leq \dfrac{Cs^{N-1}}{\|U_n\|_{L^2(B_1)}L_n} + \sum_{i=2}^k\tilde{\sigma}_{i,n}(B_{2s}),
\end{align*}
where in the last inequality we applied \eqref{eq_32} (just multiply the inequality \eqref{eq_32} by $\|U_n\|^{-1}_{L^2(B_1)}$). Therefore,
\[
s^{1-N}\sigma_{1,n}(B_s) \leq \dfrac{C}{\|U_n\|_{L^2(B_1)}L_n} + s^{1-N}\sum_{i=2}^k\tilde{\sigma}_{i,n}(B_{2s}), 
\]
for all $s \in (0,r)$. Plugging the inequalities above, we get (recall that $v_{1,n}(0) = 0$)
\begin{equation} \label{eq:39v}
  \dashint_{\partial B_r}v_{1,n} \leq C(N)\int_0^r\left[ \dfrac{C}{\|U_n\|_{L^2(B_1)}L_n} + s^{1-N}\sum_{i=2}^k\tilde{\sigma}_{i,n}(B_{2s})\right]ds. 
\end{equation}
Now, we use again the Proposition \ref{prop_Con} to estimate the last term in the inequality above:
\begin{align*}
 \int_0^rs^{1-N}\sum_{i=2}^k\tilde{\sigma}_{i,n}(B_{2s})ds & = \sum_{i=2}^k \int_0^rs^{1-N}(\Delta v_{i, n} + r_n^2\lambda_{\bar{u}_{i}}v_{i,n})(B_{2s})ds \\
 & = \sum_{i=2}^k \int_0^rs^{1-N}\Delta v_{i,n}(B_{2s})ds + \sum_{i=2}^k r_n^2\lambda_{\bar{u}_{i}}\int_0^rs^{1-N}v_{i,n}(B_{2s})ds \\
 & \leq C(N)\sum_{i=2}^k\dashint_{\partial B_r}v_{i,n} + \sum_{i=2}^k \int_0^r\frac{C(N)sr_n\lambda_{\bar{u}_{i}} \|\bar{u}_{i}\|_{L^\infty(B_1)}}{\|U_n\|_{L^2(B_1)}L_n}ds \\
 & = C(N)\sum_{i=2}^k\dashint_{\partial B_r}v_{i,n} + \sum_{i=2}^k \frac{C(N)r^2r_n\lambda_{\bar{u}_{i}} \|\bar{u}_{i}\|_{L^\infty(B_1)}}{2\|U_n\|_{L^2(B_1)}L_n}.
\end{align*}
Plugging the inequality above into \eqref{eq:39v}, and multiplying it by $|\partial B_r|$ yields to 
\begin{align*}
\int_{\partial B_r}v_{1,n} & \leq \dfrac{C(N)r^{N}}{\|U_n\|_{L^2(B_1)}L_n} + C(N)\int_{\partial B_r}v_{i,n} + \sum_{i=2}^k \frac{C(N)r_n\lambda_{\bar{u}_{i}} \|\bar{u}_{i}\|_{L^\infty(B_1)}r^{N+1}}{\|U_n\|_{L^2(B_1)}L_n},
\end{align*}
for all $r \in (0,1/2)$. Now, we integrate the inequality above with respect to $r$ to obtain
\[
\int_{B_{1/2}}v_{1,n} \leq \dfrac{C(N)}{\|U_n\|_{L^2(B_1)}L_n}  + C(N)\int_{B_{1/2}}v_{i,n} + \sum_{i=2}^k \frac{C(N)r_n\lambda_{\bar{u}_{i}} \|\bar{u}_{i}\|_{L^\infty(B_1)}}{\|U_n\|_{L^2(B_1)}L_n}. 
\]
Finally, we pass the limit as $n \to 0$, and recalling that $\|U_n\|_{L^2(B_1)} \to \infty$, $L_n \to \infty$, $r_n \to 0$, $v_{1,n} \to c_1 > 0$ and $\|v_{i,n}\|_{L^1(B_{1/2})} \to 0$, for $i=2, \ldots, k$, we get,
\[
  0 < \int_{B_{1/2}}c_1  \leq 0,
\]
which is a contradiction.
\end{proof}

It is remaining to deal with the case $m(x_n)= 1$. This is the content of the next proposition.

\begin{Proposition}\label{prop:m1}
 Under the conditions above, suppose that, up to a subsequence, $m(x_n) = 1$ for all $n$. Then \eqref{eq_cont} cannot hold true.   
\end{Proposition}

\begin{proof}
 Define $\Gamma := \{x \in \Omega\; :\; \bar{u}_1(x) = 0 \}$. We analyse two cases:\\

\noindent{\bf Case 1}: Suppose that $\frac{\mbox{dist}(x_n, \Gamma)}{r_n}$ is unbounded. Then, up to a subsequence $\frac{\mbox{dist}(x_n, \Gamma)}{r_n} \geq 1$ for large $n$. In this scenario, we can conclude that $B_{r_n}(x_n) \subset \Omega_{\bar{u}_1}$, for all $n$. Hence, by Proposition \ref{prop:equivalence_levels} and Proposition \ref{prop_cont} (or by combining Propositions 
\ref{prop_bound} and  \ref{prop_04}), $\bar{u}_1$ solves 
\[
-\Delta\bar{u}_1 = \lambda_{\bar{u}_1}\bar{u}_1 \;\;\; \mbox{ in }\;\; B_{r_n}(x_n).
\]
Hence, by applying Proposition \ref{prop_bound} and by elliptic regularity theory, we get that $\bar{u}_1 \in C^\infty(B_{r_n})$, which is a contradiction with \eqref{eq_cont}. \\

\noindent{\bf Case 2}: On the other hand, if $\frac{\mbox{dist}(x_n, \Gamma)}{r_n} <C$, for some $C > 0$, it follows from Lemma \ref{lem_dist} that we can also assume without loss of generality that $x_n \in \Gamma$, for all $n$. In view of Proposition \ref{prop:m2} and Remark \ref{rem_dist}, we can assume also that $\frac{\mbox{dist}(x_n, Z_{2})}{r_n} \geq 1$. In this case, only one component, say $\bar{u}_1$, is nonidentically zero in $B_{r_n}(x_n)$.

The proof then follows the general lines of Proposition \ref{prop:m2}, but we do not need to use the Monotonicity Lemma, since \eqref{eq_cont} is equivalent to 
\begin{equation}\label{eq_35}
\dfrac{1}{r_n^N}\int_{B_{r_n}(x_n)}|\nabla \bar{u}_1|^2 \to \infty.  
\end{equation}

We perform again the blow-up analysis. Define
\[
w_n(x) = \dfrac{1}{L_nr_n}\bar{u}_1(x_n+r_nx), \;\;\mbox{ for }\; x\in B_1,
\]
where 
\[
L_n^2 := \dfrac{1}{r_n^N}\int_{B_{r_n}(x_n)}|\nabla \bar{u}_1|^2.
\]
As before, we have
\begin{align*}
\int_{B_1}|\nabla w_n|^2 = 1,     
\end{align*}
which implies the boundness of $\int_{\partial B_1}|\nabla w_n|^2$ by Lemma \ref{lemma:bbound}. 

For simplicity, we divide the proof of this case into four steps:
\smallbreak

\noindent\textit{Step 1 (Convergence of $w_n$):} As in Proposition \ref{prop:m2}, either there exists $w_\infty \in H^1(B_1)$, such that up to a subsequence, $w_n \rightharpoonup w_\infty$ locally in $H^1(B_1)$, or the  sequence $v_n := w_n\cdot\|w_n\|_{L^2(B_1)}$ such that $v_n \to c > 0$ locally in $H^1(B_1)$.
\smallbreak

\noindent\textit{Step 2 ($w_\infty \not\equiv 0$):} It follows from the definition of $w_n$ and Proposition \ref{prop_bound} that
\[
-\Delta w_{n}(x) \leq r_n^2\lambda_{\bar{u}_1}w_n(x),
\]
for $x \in B_1$. Multiplying the inequality above by $w_{n}$ and integrating by parts, we conclude
\[
\int_{B_1}|\nabla w_{n}|^2 \leq \int_{\partial B_1}w_{n}\dfrac{\partial w_{n}}{\partial\nu} + r_n^2\lambda_{\bar{u}_{1}}\int_{B_1}w_n^2.
\]
Assume that $w_\infty \equiv 0$. Since $\|\nabla w_n\|_{L^2(\partial B_1)}$ is bounded, the right-hand side of the inequality above goes to zero, which is a contradiction.
\smallbreak

\noindent\textit{Step 3 ($w_\infty$ is a harmonic function):} As in Proposition \ref{prop:m2}, we can show that
\begin{equation}\label{eq_36}
\left\langle -\Delta w_{n} -r_n^2\lambda_{\bar{u}_1}w_{n}, \varphi\right\rangle \geq -\dfrac{C(\varphi)}{L_n}.    
\end{equation}
By taking the limit as $n \to\infty$ yields $-\Delta w_\infty \geq 0$ in $B_1$. Moreover, recall that $-\Delta w_{n} \leq r_n^2\lambda_{\bar{u}_1}w_n$ in $B_1$, we obtain $-\Delta w_\infty = 0$. Therefore, by the maximum principle, $w_\infty > 0$ in $B_1$. 

\smallbreak

\noindent\textit{Step 4 (Contradiction):} Now, we argue exactly as in Proposition \ref{prop:m2} to conclude that
\begin{align*}
\int_{\partial B_r}w_{n} & \leq \dfrac{C(N)r^{N}}{L_n} + C(N)\int_{\partial B_r}w_n + \sum_{i=2}^k \frac{C(N)r_n\lambda_{\bar{u}_{i}} \|\bar{u}_{i}\|_{L^\infty(B_1)}r^{N+1}}{L_n},   
\end{align*}
for all $r \in (0, 1/2)$. Finally, we integrate the inequality above and pass to the limit as $n \to \infty$ to conclude that
\[
0 < \int_{B_{1/2}}w_\infty \leq 0,
\]
which is a contradiction. This finishes the proof.
\end{proof}

\begin{Proposition}\label{prop:m0}
 Under the conditions above, suppose that, up to  a subsequence, $m(x_n)=0$ for all n. Then \eqref{eq_cont} cannot hold true.   
\end{Proposition}

\begin{proof}
From Remark \ref{rem_dist}, Propositions \ref{prop:m1} and \ref{prop:m2}, we can assume that $\mbox{dist}(x_n, Z_{1}) \geq r_n$. In this case, it follows from the definition of $m(x_n)$ that $\bar{u}_i\equiv 0$ in $B_{r_n}(x_n)$ for all $i=1,\ldots,k$, which is a contradiction with \eqref{eq_cont}.
\end{proof}

\section{Conclusion of the proof of the main results}\label{Sec_proof}

\begin{Lemma}\label{lemma:aux1}
Assume that problem \eqref{eigenvalue_problem2} is achieved by an optimal partition $(\omega_1,\ldots,\omega_k)\in {\mathcal P}_a(\Omega)$. Then
\[
 \sum_{i=1}^k |\omega_i|=a.
\]
\end{Lemma}

\begin{proof}
Let $(\omega_1,\ldots,\omega_k)\in {\mathcal P}_a(\Omega)$ be an optimal partition, and let $( u_1,\ldots, u_k)$ be a sequence of associated positive first eigenfunctions, which minimizes $J$ in $H_a$ by Proposition \ref{prop:equivalence_levels}. Assume by contradiction that $\sum_{i=1}^k |\omega_i|<a$. 

\smallbreak

\noindent \textbf{Claim.}  $(u_1,\ldots, u_k)\in \mathcal{S}_{\lambda_{u_1},\ldots, \lambda_{u_k}}(\Omega)$.

From \eqref{eq:subsolution}, we have
\[
-\Delta u_i \leq \lambda_{u_i} u_i \text{ in } \Omega\qquad \text{ for every $i$}.
\]
Now we prove that
\[
-\Delta \hat u_1 \geq \lambda_{u_1}u_1-\sum_{j \neq 1} \lambda_{u_j}u_j \text { in } \Omega.
\]
The latter inequality has analogous proof for $i=2,\ldots, k$. Let $\bar \varepsilon$ be such that
\begin{equation}\label{eq:contradiction_measure=a}
|B_{\bar \eps}| <a - \sum_{i=1}^k |\omega_i|.
\end{equation}
Take $x_0\in \Omega$ and $\varepsilon<\bar \varepsilon$ such that $B_{\varepsilon}(x_0)\subset \Omega$. We now check that $-\Delta \hat u_1 \geq \lambda_{u_1} u_1-\sum_{j\neq 1} \lambda_{u_j}u_j$ in $B_\eps(x_0)$. Given $\varphi\in C^\infty_c(B_\eps(x_0))$ nonnegative, for small $t>0$, we consider the deformation
\[
\tilde{u}_t=\left(\tilde{u}_{1,t}, \ldots, \tilde{u}_{k,t}\right)=\left(\frac{\left(\hat{u}_1+t \varphi\right)^{+}}{\left \|\left(\hat{u}_1+t \varphi\right)^{+}\right\|_2}, \frac{\left(\hat{u}_2-t \varphi \right)^{+}}{\left\|\left(\hat{u}_2-t \varphi\right)^+\right\|_2}, \ldots, \frac{\left(\hat{u}_k-t \varphi\right)^{+}}{\left\| \left(\hat{u}_k-t \varphi \right)^+\right\|_2}\right) .
\]
Take $\varepsilon>0$ small such that:
\[
\sum_{j=1}^k |\omega_j| + |B_\eps|<a
\]
 (recall the contradiction assumption \eqref{eq:contradiction_measure=a} and that $\omega_j = \Omega_{u_j}$, for all $j=1, \ldots, k$). Then, by Lemma \ref{def-complexa}, we have $\tilde u_t\in H_a$. We now argue exactly as in the proof of \cite[Lemma 2.1]{ContiTerraciniVerziniOPP}:
\begin{align*}
\sum_{i=1}^k \int_\Omega |\nabla u_i|^2 &=J(u_1,\ldots, u_k)\leq J(u_{1,t},\ldots, u_{k,t})= \sum_{i=1}^k \frac{\displaystyle \int_\Omega |\nabla \tilde u_{i,t}|^2}{\displaystyle \int_\Omega \tilde u_{i,t}^2}\\
&=\frac{\displaystyle \int_\Omega |\nabla (\hat u_1+t\varphi)^+|^2}{\displaystyle \int_\Omega [(\hat u_1+t\varphi)^+]^2} + \sum_{j\neq 1} \frac{\displaystyle \int_\Omega |\nabla (\hat u_{j}-t\varphi)^+|^2}{\displaystyle \int_\Omega [(\hat u_{j}-t\varphi)^+]^2}. 
\end{align*}
As $t\to 0^+$, we have by Lemma \ref{lemma:expansion_L^2}
\begin{align*}
\frac{\displaystyle \int_\Omega |\nabla (\hat u_1+t\varphi)^+|^2}{\displaystyle \int_\Omega [(\hat u_1+t\varphi)^+]^2} &= \int_{\Omega}\left|\nabla\left(\hat{u}_1+t \varphi\right)^{+}\right|^2\left(1-2 t \int_{\Omega} u_1 \varphi+\mathrm{o}(t)\right) \\
&= \int_{\Omega}\left|\nabla\left(\hat{u}_1+t \varphi\right)^{+}\right|^2-2t \int_{\Omega} u_1 \varphi \int_{\Omega}\left|\nabla u_1\right|^2+\mathrm{o}(t)
\end{align*}
and, for $j \geq 2$,
\begin{align*}
\frac{\displaystyle \int_\Omega |\nabla (\hat u_i-t\varphi)^+|^2}{\displaystyle \int_\Omega [(\hat u_i-t\varphi)^+]^2}&= \int_{\Omega}\left|\nabla\left(\hat{u}_j-t \varphi\right)^{+}\right|^2\left(1+2 t \int_{\Omega} u_j \varphi+\mathrm{o}(t)\right) \\
&= \int_{\Omega}\left|\nabla\left(\hat{u}_j-t \varphi\right)^{+}\right|^2+2t \int_{\Omega} u_j \varphi \int_{\Omega}\left|\nabla u_j\right|^2+\mathrm{o}(t).
\end{align*}
Therefore, using the fact that $u_iu_j \equiv 0$ for $i\neq j$, and   $\left(\hat{u}_1+t \varphi\right)^{+}+$ $\sum_{j \geqslant 2}\left(\hat{u}_j-t \varphi\right)^{+}=|\hat{u}_1+t \varphi|$ , we have
$$
\begin{aligned}
\sum_{i=1}^k \int_{\Omega} |\nabla u_i|^2 &\leq  \sum_{i=1}^k \int_{\Omega} |\nabla (\hat{u}_i+t \varphi)^+|^2-2 t \int_{\Omega} \lambda_{u_1} u_1 \varphi+2 t \sum_{j \geqslant 2} \int_{\Omega} \lambda_{u_j} u_j \varphi+\mathrm{o}(t) \\
&=\sum_{i=1}^k \int_{\Omega}|\nabla u_i|^2+2 t \int_{\Omega}\left(\nabla \hat{u}_1\cdot \nabla \varphi-\left(\lambda_{u_1} u_1-\sum_{j \geqslant 2} \lambda_{u_j} u_j\right) \varphi\right)+\mathrm{o}(t)
\end{aligned}
$$
as $t \rightarrow 0^+$, and hence
$$
\int_{\Omega}\left(\nabla \hat{u}_1\cdot \nabla \varphi-\left( \lambda_{u_1} u_1-\sum_{j \geqslant 2} \lambda_{u_j} u_j\right) \varphi\right) \geq 0 .
$$
Therefore the claim holds true.

\smallbreak

\noindent \textit{Conclusion of the proof.} Since $(u_1,\ldots, u_k)\in \mathcal{S}_{\lambda_{u_1},\ldots, \lambda_{u_k}}(\Omega)$, then $\Gamma_{u}$ has zero measure by Lemma \ref{lemma:classeS}.  Therefore $|\Omega|= |\Omega\setminus \Gamma_u|=|\cup_{i=1}^k \Omega_{u_i}|=\sum_{i=1}^k |\Omega_{u_i}|\leq a<|\Omega|$, a contradiction.
\end{proof}

\smallbreak

\begin{Lemma}\label{lemma:aux2} Assume that problem \eqref{eigenvalue_problem2} is achieved by an optimal partition $(\omega_1,\ldots,\omega_k)\in {\mathcal P}_a(\Omega)$. Then:
\[\omega_i \text{ is connected for every  } i=1,\ldots,k.
\]
\end{Lemma}
\begin{proof}
\noindent Assume by contradiction that $\omega_1$ is not connected, and let $A_1$ be a connected component of $\omega_1$ such that $\lambda_1(\omega_1)=\lambda_1(A_1)$. Observe that $\omega_1\setminus A_1$ is open and nonempty. Then $(A_1,\omega_2,\ldots, \omega_k)\in {\mathcal P}_a(\Omega)$ satisfies
\[
\lambda_1(A_1)+\sum_{i=2}^k \lambda_1(\omega_i)=c_a,
\]
\textit{i.e.} it is an optimal partition. On the other hand,
\[
 |A_1|+\sum_{i=2}^k |\omega_i|<\sum_{i=1}^k |\omega_i|=a,
\]
which contradicts the previous lemma.
\end{proof}

\begin{proof}[{Proof of Theorem \ref{thm:main} completed}]
By Proposition \ref{thm:condition_C1}, we have the existence of $(u_1,\ldots, u_k)\in H_a$ which minimizes $J$ over the set $H_a$, that is, the level $\tilde c_a$ is achieved. As it is proved in Section \ref{secLip}, each $\bar u_i$ is locally Lipschitz continuous. Therefore, by Proposition \ref{prop:equivalence_levels}, \eqref{eigenvalue_problem2} has a solution, and problems \eqref{eigenvalue_problem2} and \eqref{main_func} are equivalent. Moreover, from Proposition \ref{lemma:aux1}, we have that solutions to \eqref{eigenvalue_problem2} are also minimizers to \eqref{eigenvalue_problem}, and these levels coincide. 
\end{proof}

\begin{proof}[Proof of Theorem \ref{th-bolas}]
The proof is a consequence of the Faber-Krahn inequality: given an open set $\omega\subset \R^N$, then $\lambda_1(\omega)\geq \lambda_1(\omega^*)$, where $\omega^*$ is an open ball such that $|\omega^*|=|\omega|$; moreover, equality is achieved if and only if $\omega$ is a ball.

Let $(\omega_1\ldots, \omega_k)\in \mathcal{P}_a(\Omega)$ be an optimal partition for problem \eqref{eigenvalue_problem2} (which exists, by Theorem \ref{thm:main}). Let $R_{r_i}$ be an open ball such that $|B_{r_i}|=|\omega_i|$, for each $i$. If $a$ is sufficiently small, then we can assume that 
\[
B_{r_i}\cap B_{r_j}=\emptyset \quad \forall i\neq j,\qquad \text{ and } \qquad \cup_{i=1}^k B_{r_i}\subset \Omega.
\]
By the Faber-Krahn inequality we have that $(B_{r_1},\ldots, B_{r_k})$ is an optimal partition and, up to translation of the center, we may assume that $\omega_i=B_{r_i}$. 

We now claim that $r_1=\ldots=r_k$, which finishes the proof. But this is a consequence of the fact that the function:
\[
(r_1,\ldots,r_k)\in (\R^+)^k\mapsto \sum_{i=1}^k \lambda_1(B_{r_i})=\lambda_1(B_1) \sum_{i=1}^k \frac{1}{r_i^2}
\]
admits a unique minimizer on the set
\[
\left\{(r_1,\ldots,r_k)\in \R^k: \sum_{i=1}^k |B_{r_i}|=a\right\}=\left\{(r_1,\ldots,r_k)\in \R^k: |B_1|\sum_{i=1}^k {r_i}^N=a\right\}
\]
precisely at a point where $r_1=\ldots=r_k$, by the Lagrange multipliers rule.
\end{proof}

\begin{proof}[Proof of Theorem \ref{th-simetria}]
Let $(\omega_1, \omega_2)$ be a solution of \eqref{eigenvalue_problem} with corresponding eigenfunctions $u_1, u_2$. Then, consider $\omega_1^*$ and $\omega_2^*$ the cap symmetrization of $\omega_1$ and $\omega_2$, with respect $e$ and $-e$, respectively. Then $(\omega_1^*, \omega_2^*)\in {\mathcal P}_a(\Omega)$ and, since $\lambda_1^*(\omega_i^*) \leq \lambda_1(\omega_i)$, for $i=1,2$, see \cite[Section 7.5]{BaernsteinBook}, then $(\omega_1^*, \omega_2^*)$ solves \eqref{eigenvalue_problem}. Moreover, the positive first eigenfunctions in $\omega_1^*$ and $\omega_2^*$ are foliated Schwarz symmetric with respect to $e$ and $-e$, respectively. 
\end{proof}

\appendix

\section{Auxiliary results} \label{appendix}

\subsection{Deformations}

Here we collect some results regarding the deformations used in the paper. This type of deformation appears in the context of multiphase optimal shape problems without the volume constraints. We refer the reader to \cite{ContiTerraciniVerziniAsymptotic,ContiTerraciniVerziniOPP,ContiTerraciniVerziniVariation}.

\begin{Lemma}\label{lemma:expansion_L^2} Let $u \in L^2(\Omega)$ with $u^+\not\equiv 0$. Then, for all $\varphi \in L^2(\Omega)$, 
\begin{equation}\label{inverse-norm-lemma}
\displaystyle \frac{1}{\|(u \pm t\varphi)^{+}\|_2^2 }= \dfrac{1}{\|u^+\|_2^2} \mp \frac{2t}{\|u^+\|_2^4} \int_{\Omega}u^+\varphi + O(1)\,  \|\varphi\|_2^2 \, t^2\qquad \text{ as } t\to 0^+,\vspace{10pt}
\end{equation}
where $O(1)$ depends only on $\|u^+\|_ 2$.
\end{Lemma}
\begin{proof}
Observe that
\begin{align*}
 \dfrac{1}{\|(u \pm t\varphi)^{+}\|_2^2 } &- \dfrac{1}{\|u^+\|_2^2} = \displaystyle \dfrac{\int_{u>0 \cap \Omega_{\varphi}}u^2 - \int_{u\pm t \varphi>0\cap \Omega_{\varphi}} u^2 \pm 2t u \varphi + t^2 \varphi^2}{\|(u \pm t\varphi)^{+}\|_2^2 \, \|u^+\|_2^2}  \\
&= \displaystyle\dfrac{\int_{u>0 \cap \Omega_{\varphi}}u^2 - \int_{u\pm t \varphi>0\cap \Omega_{\varphi}} u^2 \mp 2t \int_{u\pm t \varphi>0\cap \Omega_{\varphi}}  u \varphi - t^2\int_{u\pm t \varphi>0\cap \Omega_{\varphi}} \varphi^2}{\|(u \pm t\varphi)^{+}\|_2^2 \, \|u^+\|_2^2}  \\
&=\displaystyle\dfrac{\int_{0< u \leq \mp t \varphi \cap\{ \pm \varphi <0 \}}u^2 - \int_{\mp t \varphi< u < 0 \cap\{ \pm \varphi >0 \}}u^2 - t^2\int_{u\pm t \varphi>0\cap \Omega_{\varphi}} \varphi^2  }{\|(u \pm t\varphi)^{+}\|_2^2 \, \|u^+\|_2^2} \mp 2t \dfrac{\int_{u\pm t \varphi>0\cap \Omega_{\varphi}}  u \varphi}{\|(u \pm t\varphi)^{+}\|_2^2 \, \|u^+\|_2^2}  \\
& =  \mp 2t \dfrac{\int_{u>0\cap \Omega_{\varphi}}  u \varphi}{\|(u \pm t\varphi)^{+}\|_2^2 \, \|u^+\|_2^2} \mp 2t \dfrac{\int_{u\pm t \varphi>0\cap \Omega_{\varphi}}  u \varphi- \int_{u>0\cap \Omega_{\varphi}}  u \varphi}{\|(u \pm t\varphi)^{+}\|_2^2 \, \|u^+\|_2^2} + O(1)\, \|\varphi\|^2_{2} \, t^2\\
&=\displaystyle \mp \dfrac{2 t}{\|u^+\|_2^4 } \int_{\Omega}u^+ \varphi + O(1)\, \|\varphi\|^2_{2} \, t^2 \qquad \text{ as } t\to 0^+.  \qedhere
\end{align*}
\end{proof}

\begin{Lemma}\label{def-simples}
Let $u_1,\ldots, u_k \in L^2(\Omega)$ be nonnegative functions such that $u_i\cdot u_j\equiv 0$ for all $i\neq j$, $\|u_i\|_{L^2}=1$ for all $i=1, \ldots, k$, and $\varphi\in C^\infty_c(\Omega)$ be a nonnegative function. Consider, for $t>0$ small, the deformation
\begin{equation}
\check{u}_{t}=(\check{u}_{1,t}, \check{u}_{2,t}, \ldots, \check{u}_{k,t})=\left(\frac{(u_1 - t\varphi)^+}{\|(u_1 - t\varphi)^+\|_2}, u_2, \ldots, u_k\right) .
\end{equation}
Then:
\begin{enumerate}[i)]
\item $\displaystyle \int_\Omega \check{u}_{i,t}^2=1$ for every $i$;
\item $\Omega_{\check u_{i,t}}\subseteq \Omega_{u_i}$ for all $i\geq 1$;
\item $\check u_{i,t}\cdot  \check u_{j,t} \equiv 0$ for $i\neq j$.
\end{enumerate}
\end{Lemma}
\begin{proof}
It is obvious that i) holds and iii) is a consequence of ii). Regarding ii), observe that
\[
x\notin \Omega_{u_i} \implies  \check u_{1,t} (x)=0. \qedhere
\]
\end{proof}

Recall that, for $u_1,\ldots,u_k$ such that $u_i\cdot u_j\equiv 0$ for all $i\neq j$, we denote:
\[
\hat{u}_i=u_i-\sum_{j\neq i} u_j.
\]
\begin{Lemma}\label{def-complexa}
Let $u_1,\ldots, u_k \in L^2(\Omega)$ be such that $u_i\cdot u_j\equiv 0$ for all $i\neq j$, and $u_i\geq 0$ and $\|u_i\|_{L^2}=1$ for every $i$. Take $\mathcal{A}\Subset \Omega$, and let $\varphi\in C^\infty_c(\mathcal{A})$ be a nonnegative function. Consider, for $t>0$ small, the deformation
\[
\tilde{u}_t=\left(\tilde{u}_{1,t}, \ldots, \tilde{u}_{k,t}\right)=\left(\frac{\left(\hat{u}_1+t \varphi\right)^{+}}{\left \|\left(\hat{u}_1+t \varphi\right)^{+}\right\|_2}, \frac{\left(\hat{u}_2-t \varphi \right)^{+}}{\left\|\left(\hat{u}_2-t \varphi\right)^+\right\|_2}, \ldots, \frac{\left(\hat{u}_k-t \varphi\right)^{+}}{\left\| \left(\hat{u}_k-t \varphi \right)^+\right\|_2}\right) .
\]
Then:
\begin{enumerate}[i)]
\item $\displaystyle \int_\Omega \tilde u_{i,t}^2=1$ for every $i$;
\item $\Omega_{\tilde u_{1,t}}\subseteq \Omega_{u_1}\cup \mathcal{A}$, and $\Omega_{\tilde u_{i,t}}\subseteq \Omega_{u_i}$ for $i>1$. 
\item $\tilde u_{i,t}\cdot  \tilde u_{j,t} \equiv 0$ for $i\neq j$. 
\end{enumerate}
\end{Lemma}
\begin{proof} The statement i) is obviously true. Regarding ii), we have
\[
\hat u_1(x)+t\varphi(x)>0 \implies u_1(x)+t\varphi(x)>\sum_{j\neq 1} u_j(x)\geq 0 \implies u_1(x)>0 \text{ or } \varphi(x)>0.
\]
For $i>1$,
\[
 \hat u_i(x)-t\varphi(x)>0 \implies u_i(x)>\sum_{j\neq i} u_j(x)+t\varphi(x)\geq 0 \implies u_i(x)>0.
\]

Finally, for iii), since $u_i\cdot u_j \equiv 0$ for $i \neq j$, it is obvious that $\tilde u_{i,t} \cdot \tilde u_{j,t} \equiv 0$ for $i\neq j$ with $i, j \geq 2$. Now, by contradiction, suppose that there exists $x\in \Omega_{\tilde u_{1,t}}\cap \Omega_{\tilde u_{i,t}}\neq \emptyset$, for some $i>1$. Then

\begin{align*}
\hat u_1(x)+t\varphi(x)>0 \text{ and }  \hat u_i(x)-t\varphi(x)>0 & \iff \sum_{j\neq 1} u_j(x) <u_1(x)+t\varphi(x)< u_i(x)-\sum_{j\neq 1,i}u_j(x)\\
			&\implies u_i(x) > u_i(x) + 2 \sum_{j\neq 1,i} u_j(x) = u_i(x),
\end{align*}
a contradiction. Notice that for the last equality we have used that $u_i\cdot u_j \equiv 0$ for $i \neq j$ and $u_i \geq 0$ for every $i$.

\end{proof}

\subsection{Auxiliary lemmas}

We recall the following Liouville type theorem for subharmonic functions.
\begin{Proposition}\label{prop:Liouville}
 Assume that $u_1,\ldots, u_k\in  H^1_{loc}(\R^N)\cap C(\R^N)$ are nonnegative  subharmonic functions such that $u_i\cdot u_j\equiv 0$ in $\R^N$. Assume moreover that $u_1,\ldots, u_k$ are bounded.
Then all functions but possibly one function are trivial.
\end{Proposition}
\begin{proof}
The result follows directly from \cite[Proposition 2.2]{NTTV1} applied with $\alpha=0$. We observe that, even though this case is not stated in the proposition, the proof is exactly the same.
\end{proof}

The next useful inequality is used to prove the continuity of the solutions and its proof can be found in \cite{GiTr, GS1996}.

\begin{Proposition}\label{prop_Con}
Let $B_{r_0}(x_0) \subset \Omega$, $u \in H^1(B_{r_0}(x_0))$ and suppose $\Delta u$ is a measure satisfying 
\begin{equation}\label{eq_22}
\int_0^rs^{1-N}\left[\int_{B_s(x_0)} d |\Delta u|\right]ds < + \infty,
\end{equation}
for all $r \in (0,r_0)$. Then, the limit $\displaystyle\lim_{\rho}\dashint_{\partial B_{\rho}(x_0)}u$ exists, and we can define
\[
u(x_0) = \lim_{\rho}\dashint_{\partial B_{\rho}(x_0)}u.
\]
In addition, for all $r \in (0, r_0)$ 
\begin{equation}\label{eq_con}
\dashint_{\partial B_r(x_0)} [u- u(x_0)] = C(N)\int_0^r s^{1-N}\left[\int_{B_s(x_0)} d\Delta u \right]ds.
\end{equation}
The inequality in \eqref{eq_22} is also true in the case where $u \in L^{\infty}(B_{r_0}(x_0))$, and there exists $f \in  L^{\infty}(B_{r_0}(x_0))$ such that $-\Delta u^+ \leq f$ and $-\Delta u^- \leq f$.
\end{Proposition}

In the proof of Lipschitz regularity, we make use of the Caffarelli-Jerison-Kenig Monotonicity Lemma that we state next. For a proof of this result, we refer the reader to \cite{CJK2002} (see also \cite{ContiTerraciniVerziniVariation,Velichkovpaper}).

\begin{Lemma}[Caffarelli-Jerison-Kenig monotonicity lemma]\label{jck_lemma} Let $u_1, u_2 \in H_0^1(B_{r_0}(x_0))\cap L^\infty(B_{r_0}(x_0))$ with $u_1\cdot u_2 =0$ a.e. in $\Omega$. Suppose that, for some constant $\gamma \geq 0$, 
\[
\Delta u_1 \geq - \gamma \;\;\mbox { and } \;\; \Delta u_2 \geq - \gamma \;\; \mbox{ in }\; B_{r_0}(x_0).
\]
Set 
\[
\psi(r) := \left(\dfrac{1}{r^2}\int_{B_r(x_0)}\dfrac{|\nabla u_1|^2}{|x-x_0|^{N-2}}\right) 
\left(\dfrac{1}{r^2}\int_{B_r(x_0)}\dfrac{|\nabla u_2|^2}{|x-x_0|^{N-2}}\right).
\]
\smallskip
Then, there exists a constant $C > 0$ such that $\psi(r) \leq C$ for all $r \in (0, r_0/2)$.
\end{Lemma}

\bigskip
\noindent {\bf Acknowledgements:} P\^edra D. S. Andrade is partially supported by CAPES-INCTMat - Brazil and by the Portuguese government through FCT - Funda\c c\~ao para a Ci\^encia e a Tecnologia, I.P., under the projects UID/MAT/04459/2020. Hugo Tavares is partially supported by the Portuguese government through FCT - Funda\c c\~ao para a Ci\^encia e a Tecnologia, I.P., under the projects UID/MAT/04459/2020 and PTDC/MAT-PUR/1788/2020. Ederson Moreira dos Santos is partially supported by CNPq grant 309006/2019-8. Makson S. Santos is partially supported by FAPESP
grant 2021/04524-0 and by the Portuguese government through FCT- Funda\c c\~ao para a Ci\^encia e a Tecnologia, I.P., under the projects UID/MAT/04459/2020 and PTDC/MAT-PUR/1788/2020. This study was financed in part by the Coordena\c c\~ao de Aperfei\c coamento de Pessoal de N\'ivel Superior - Brazil (CAPES) -
Finance Code 001.

The authors would like to thank Dario Mazzoleni for some interesting discussions and for pointing out some references.

\medbreak

\bigskip

\noindent\textsc{P\^edra D. S. Andrade, Makson S. Santos and Hugo Tavares}\\
Departamento de Matem\'atica do Instituto Superior T\'ecnico\\
Universidade de Lisboa\\
1049-001 Lisboa, Portugal\\
\noindent\texttt{pedra.andrade@tecnico.ulisboa.pt, makson.santos@tecnico.ulisboa.pt, \\hugo.n.tavares@tecnico.ulisboa.pt}

\vspace{.15in}

\noindent\textsc{Ederson Moreira dos Santos}\\
Instituto de Ci\^encias Matem\'aticas e de Computa\c c\~ao\\
Universidade de S\~ao Paulo -- USP\\
13566-590, Centro, S\~ao Carlos - SP, Brazil\\
\noindent\texttt{ederson@icmc.usp.br}

\end{document}